\documentclass[11pt,reqno,tbtags,a4paper]{amsart}
\usepackage{amssymb}
\usepackage{url}
\usepackage[square,numbers]{natbib}
\bibpunct[, ]{[}{]}{;}{n}{,}{,}

\title
{Graphons and cut metric on $\sigma$-finite measure spaces}

\date{5 August, 2016; revised 16 August 2016}

\author{Svante Janson}
\address{Department of Mathematics, Uppsala University, PO Box 480,
SE-751~06 Uppsala, Sweden}
\email{svante.janson@math.uu.se}
\urladdr{http://www.math.uu.se/svante-janson}


\numberwithin{equation}{section}

\renewcommand\le{\leqslant}
\renewcommand\ge{\geqslant}

\allowdisplaybreaks

\theoremstyle{plain}
\newtheorem{theorem}{Theorem}[section]
\newtheorem{lemma}[theorem]{Lemma}
\newtheorem{proposition}[theorem]{Proposition}
\newtheorem{corollary}[theorem]{Corollary}

\theoremstyle{definition}
\newtheorem{example}[theorem]{Example}
\newtheorem{definition}[theorem]{Definition}

\newtheorem{remark}[theorem]{Remark}

\theoremstyle{remark}

\newenvironment{romenumerate}[1][-10pt]{
\addtolength{\leftmargini}{#1}\begin{enumerate}
 \renewcommand{\labelenumi}{\textup{(\roman{enumi})}}%
 }{\end{enumerate}}

\newenvironment{PXenumerate}[1]{
\begin{enumerate}
 \renewcommand{\labelenumi}{\textup{(#1\arabic{enumi})}}%
 }{\end{enumerate}}

\newenvironment{PQenumerate}[1]{
\begin{enumerate}
 \renewcommand{\labelenumi}{\textup{(#1)}}%
 }{\end{enumerate}}

\newcounter{oldenumi}
\newenvironment{romenumerateq}
{\setcounter{oldenumi}{\value{enumi}}
\begin{romenumerate} \setcounter{enumi}{\value{oldenumi}}}
{\end{romenumerate}}

\newcounter{thmenumerate}
\newenvironment{thmenumerate}
{\setcounter{thmenumerate}{0}%
 \def\item{\par
 \refstepcounter{thmenumerate}\textup{(\roman{thmenumerate})\enspace}}
}
{}

\newcounter{xenumerate}   

\newcommand\pfitemx[1]{\par#1:}
\newcommand\pfitemref[1]{\pfitemx{\ref{#1}}}

\newcommand\step[2]{\smallskip\noindent\emph{Step #1: #2} \noindent}

\newcommand{\refT}[1]{Theorem~\ref{#1}}
\newcommand{\refC}[1]{Corollary~\ref{#1}}
\newcommand{\refL}[1]{Lemma~\ref{#1}}
\newcommand{\refR}[1]{Remark~\ref{#1}}
\newcommand{\refS}[1]{Section~\ref{#1}}
\newcommand{\refSS}[1]{Section~\ref{#1}}
\newcommand{\refP}[1]{Proposition~\ref{#1}}
\newcommand{\refD}[1]{Definition~\ref{#1}}
\newcommand{\refE}[1]{Example~\ref{#1}}

\newcommand\nopf{\qed}   

\newcommand{\sumko}{\sum_{k=0}^\infty}

\newcommand{\sumk}{\sum_{k=1}^\infty}

\newcommand\set[1]{\ensuremath{\{#1\}}}

\newcommand\xpar[1]{(#1)}
\newcommand\bigpar[1]{\bigl(#1\bigr)}

\newcommand\biggpar[1]{\biggl(#1\biggr)}

\newcommand\xcpar[1]{\{#1\}}

\newcommand\abs[1]{|#1|}
\newcommand\bigabs[1]{\bigl|#1\bigr|}
\newcommand\Bigabs[1]{\Bigl|#1\Bigr|}

\newcommand\lrabs[1]{\left|#1\right|}
\def\rompar(#1){\textup(#1\textup)}    
\newcommand\xfrac[2]{#1/#2}

\def\xexp(#1){e^{#1}}

\newcommand\frax[1]{\{#1\}}

\newcommand\ntoo{\ensuremath{{n\to\infty}}}
\newcommand\Ntoo{\ensuremath{{N\to\infty}}}

\newcommand\mtoo{\ensuremath{{m\to\infty}}}

\newcommand\ttoo{\ensuremath{{t\to\infty}}}

\newcommand\norm[1]{\|#1\|}
\newcommand\bignorm[1]{\bigl\|#1\bigr\|}
\newcommand\Bignorm[1]{\Bigl\|#1\Bigr\|}

\newcommand\punkt{.\spacefactor=1000}    
    
\newcommand\ie{i.e\punkt}
\newcommand\eg{e.g\punkt}
\newcommand\viz{viz\punkt}
\newcommand\cf{cf\punkt}
\newcommand{\as}{a.s\punkt}
\newcommand{\aex}{a.e\punkt}

\newcommand{\tend}{\longrightarrow}

\newcommand\vto{\overset{\mathrm{v}}{\tend}}
\newcommand\wto{\to}
\newcommand\wxto{\overset{\mathrm{w}*}{\tend}}
\newcommand\asto{\overset{\mathrm{a.s.}}{\tend}}
\newcommand\eqd{\overset{\mathrm{d}}{=}}

\newcommand\bbR{\mathbb R}

\newcommand\bbN{\mathbb N}

\newcounter{CC}
\newcounter{cc}

\newcommand\E{\operatorname{\mathbb E{}}}

\newcommand\Po{\operatorname{Po}}

\newcommand\supp{\operatorname{supp}}
\newcommand\sgn{\operatorname{sgn}}

\newcommand\gd{\delta}

\newcommand\gf{\varphi}

\newcommand\gG{\Gamma}

\newcommand\gs{\sigma}

\newcommand\eps{\varepsilon}

\renewcommand\phi{\xxx}  

\newcommand\cA{\mathcal A}
\newcommand\cB{\mathcal B}

\newcommand\cE{\mathcal E}
\newcommand\cF{\mathcal F}

\newcommand\cP{\mathcal P}

\newcommand\cS{{\mathcal S}}

\newcommand\cW{\mathcal W}

\newcommand\tG{\tilde G}

\newcommand\tS{{\tilde S}}

\newcommand\tV{\tilde V}
\newcommand\tW{\widetilde W}

\newcommand\bW{\overline W}

\newcommand\ett[1]{\boldsymbol1{\xcpar{#1}}}

\newcommand\etta{\boldsymbol1}

\newcommand\qw{^{-1}}
\newcommand\qww{^{-2}}
\newcommand\qq{^{1/2}}
\newcommand\qqw{^{-1/2}}

\newcommand\oi{\ensuremath{[0,1]}}

\newcommand\ooo{[0,\infty)}
\newcommand\ooox{[0,\infty]}

\newcommand\dd{\,\mathrm{d}}

\newcommand{\ui}{uniformly integrable}
\newcommand{\sui}{semiuniformly integrable}

\newcommand\rhs{right-hand side}

\newcommand\cn[1]{\norm{#1}\cut}
\newcommand\bigcn[1]{\bignorm{#1}\cut}
\newcommand\Bigcn[1]{\Bignorm{#1}\cut}
\newcommand\cnx[2]{\norm{#2}_{\square,#1}}
\newcommand\cut{_{\square}}
\newcommand\dcut{\delta_{\square}}
\newcommand\dl{\delta_{1}}
\newcommand\dlp{\delta_{p}}
\newcommand\mpp{measure-preserving}
\newcommand\mpm{\mpp{} map}
\newcommand\mpb{\mpp{} bijection}

\newcommand{\ps}{probability space}

\newcommand\equ{\cong} 
\newcommand\tensor{\otimes}
\newcommand\itensor{\check\otimes}
\newcommand\ptensor{\hat\otimes}
\newcommand\leb{\lambda}
\newcommand\lebb{\leb^2}
\newcommand\gsa{$\gs$-algebra}
\newcommand\bbRp{\bbR_+}
\newcommand\tmu{\tilde\mu}
\newcommand\hmu{\hat\mu}
\newcommand\tnu{\tilde\nu}
\newcommand\ZZ{Z}

\newcommand\sss{{\Omega}}
\newcommand\weakx{weak${}^*$}
\newcommand\gsf{$\sigma$-finite}
\newcommand\qpi[1]{^{\pi_{#1}}}
\newcommand\qphi[1]{^{\gf_{#1}}}
\newcommand\qpsi[1]{^{\psi_{#1}}}
\newcommand\qgf{^{\gf}}
\newcommand\hmux{\tmu}
\newcommand\CI{C_{\oi}}
\newcommand\CcI{C_{c,\oi}}
\newcommand\tg{\tilde g}
\newcommand\thx{\tilde h}
\newcommand\tr{\tilde r}
\newcommand\muoo{\mu_\infty}
\newcommand\KxK{{K\times K}}
\newcommand\SxS{{S\times S}}

\newcommand\cnc[1]{\norm{#1}_{C}}
\newcommand\norml[1]{\norm{#1}_{L^1}}

\newcommand\normlp[1]{\norm{#1}_{L^p}}
\newcommand\normloo[1]{\norm{#1}_{L^\infty}}
\newcommand\qx[1]{#1\times#1}
\newcommand\qxp[1]{(#1\times#1)}
\newcommand\qxq[1]{#1^2}
\newcommand\gfgf{(\gf_1,\gf_2)}
\newcommand\eeq{elementarily equivalent}
\newcommand\psiq{\psi^{\otimes 2}}
\newcommand\sz{^{\mathsf s}}
\newcommand\str[2]{\Upsilon_{#2}#1}
\newcommand\strp[2]{\str{(#1)}{#2}}
\newcommand\stri[2]{\Upsilon^{(1)}_{#2}#1}

\newcommand\strii[2]{\Upsilon^{(2)}_{#2}#1}
\newcommand\striip[2]{\strii{(#1)}{#2}}
\newcommand\xW{\overline W}
\newcommand\restr[1]{|_{#1}}
\newcommand\restrq[1]{\restr{\qx{#1}}}
\newcommand\NN{^{(N)}}
\newcommand\MM{^{(M)}}
\newcommand\lsmu{\ensuremath{L^1(S,\mu)}}

\newcommand\xV{W}
\newcommand\hU{\bar U}
\newcommand\urt{uniformly regular tails}
\newcommand\ucrt{uniformly cut regular tails}
\newcommand\ucr{upper cut regular}
\newcommand\cWpc{\mathcal W(p,C)}
\newcommand\VV{V^*}
\newcommand\Uc{U^{\textsf c}}
\newcommand\ettaqx[1]{\etta_{\qx{#1}}}
\newcommand\hW{\widehat W}
\newcommand\chW{\check W}
\newcommand\hS{\hat S}
\newcommand\chS{\check S}

\newcommand{\Holder}{H\"older}

\newcommand\ER{Erd\H os--R\'enyi}
\newcommand{\Lovasz}{Lov\'asz}

\begin{document}

\begin{abstract} 
Borgs, Chayes, Cohn and Holden (2016+) recently extended the definition of
graphons from probability spaces to arbitrary $\sigma$-finite measure
spaces, in order to study limits of sparse graphs.
They also extended the definition of the cut metric, and proved various
results on the resulting metric space.

We continue this line of research and give various further results on
graphons and the cut metric in this general setting, extending known results
for the standard case of graphons on probability spaces.
In particular, we characterize pairs of equivalent graphons, and we give new
results on completeness and compactness.
\end{abstract}

\maketitle

\section{Introduction}\label{Sintro}

The theory of graph limits and graphons has become a successful tool to
study large dense graphs.
First, any sequence of graphs, with orders tending to infinity, has at least
a subsequence that converges to a 
\emph{graph limit}, which can be represented (non-uniquely) by a \emph{graphon},
which in this context is a \oi-valued symmetric function defined on $\qx S$
where $S$ is a probability space (that often is taken to be \oi).
Secondly, any such graphon $W$ defines a sequence of random graphs $G(n,W)$,
which gives a large family of dense random graphs with different properties.
See \eg{} \citet{LSz}, \citet{BCLSV1,BCLSV2}, \citet{Austin}, \citet{SJ209} and
\citet{Lovasz}.

There have been several partial extensions of the theory to sparse graphs,
using more general graphons.
\citet{BR09} considered graphons that are bounded (but not necessarily
\oi-valued), and this was extended by \citet{BCCZ14a} to unbounded graphons,
assuming that the graphons are integrable (and  usually in $L^p$ for some
$p>1$). 
These papers also consider signed graphons (in connection with weighted
graphs where the weigths may be negative).

Another leap in increasing generality was taken by
\citet{VR} and \citet{BCCH16}, with some special cases studied by
\citet{CaronFox} and \citet{HSM15}; 
the new idea is to let the graphons be defined on $\qx S$ for an arbitrary
\gsf{} measure space $S$ (and not just a \ps, as earlier);
it turns out that without loss of generality, the measure space $S$ can be
taken to be $\bbRp$ with Lebesgue measure \cite[Proposition 2.8]{BCCH16}.
(Only this case is considered in \cite{VR}.)
The graphons in \cite{VR} and \cite{BCCH16} are mainly \oi-valued and
generate random graphs by the construction described in \refSS{SSRG} below;
however,
\cite{BCCH16} considers also unbounded and signed graphons (that may occur as
limits of weighted graphs).
(The version of the construction in \cite{VR} also includes additional stars
and isolated edges; we do not treat these parts in the present paper.)

\citet{VR} is focussed on properties of the resulting random graphs, and in
particular the fact that, as a consequence of results by
\citet{Kallenberg1990,Kallenberg-symmetries}, all random graphs that are
exchangeable in a
certain sense can be obtained in this way.
\citet{BCCH16} contains related results on exchangeable random graphs, and
also many results on convergence of graphs and graphons in the cut metric
$\dcut$, as well as some results for the related metrics $\gd_1$ and $\gd_p$.

The present paper is mainly inspired by \citet{BCCH16}, and gives various
further results on convergence in the cut metric for  
(unbounded, possibly
signed) graphons defined on \gsf{} measure spaces.
We also give some related results for the metrics $\gd_1$ and $\gd_p$.
The results should be compared to the corresponding results for standard
graphons on \ps{s} in \cite{SJ249}.

Sections \ref{Sdef}--\ref{Stop} contain definitions,  some
earlier results and other preliminaries.

\refS{Seq} extends a result by \citet{BR09} to the present generality and
shows that for  Borel spaces,
the infimum in the definition of the cut distance is attained
(\refT{T1}).
This leads to a characterisation  (\refT{T=0}) of equivalent graphons on such
spaces as
having a pair of pullbacks that are \aex{} equal, also extending a result by
\cite{BR09}, and another, completely general, 
characterisation of equivalence (\refT{T=}) as
being generated by pull-backs and trivial extensions, extending
\cite[Theorem 8.3]{SJ249}. Several consequences of the latter
characterisation are also given.

\refS{Scomplete} gives results on completeness of sets for the cut metric
(\refT{TB}), after some preliminary results for the cut norm.
Several counter examples are also given, illustrating the conditions in the
theorem; the set of all graphons is, unfortunately, not complete.

Sections \ref{Scompact}--\ref{ScompactPf} give results on (relative)
compactness in the cut metric
that extend and improve results in \cite{BCCH16}.
We give a complete characterisation of totally bounded sets
(\refT{TD}); however, since we do not have a complete characterisation of
complete sets, we do not obtain a complete characterisation of 
(relatively) compact sets
of graphons without adding extra conditions (for example Theorems
\ref{TUC}--\ref{TC}).

\begin{remark}
  The present paper thus studies the cut metric for graphons on \gsf{}
  measure spaces.
Since graphs may be represented by graphons, this includes results on
convergence of graphs to graphons in this sense, see \refSS{SSWG}.
 Note, however, that the cut metric is only one of several
  conceivable metrics (or other ways of defining limits), 
see \eg{} \cite{BCLSV1}, \cite{BCLSV2}, \cite{BR09}, \cite{BCCH16}, \cite{VR2}.
In the standard case of \oi-valued graphons on \ps{s}, 
a number of different metrics and topologies are equivalent, basically
because they define compact topologies that are comparable and thus equal,
see \eg{} \cite{BCLSV1,BCLSV2}.
In extensions like the one treated here, compactness is lost, and there is
no reason to expect various notions to be equivalent,
although there are some partial results under extra assumptions, see \eg{}
\cite{BR09}.
On the contrary, there
are counter examples, see for example \cite[Proposition 2.24(iv)]{BCCH16},
showing that different notions of convergence are not equivalent.

The cut metric has been hugely successful in the standard setting, but it is
not at all clear that it is of equal importance in extensions like the one
studied here. (For one thing, the fact that the metric is not complete on
the space of graphons studied here, see \refS{Scomplete}, is a warning that
the definitions may be not optimal. 
Moreover,  \cite{VR} considers also some non-integrable graphons, although
the definition of the cut metric requires integrability.)
Nevertheless, the present paper considers  exclusively the cut
metric (and the related $\gd_1$ and $\gd_p$), 
hoping that this will inspire future studies of  other metrics and
modes of convergence for general graphons and (sparse) graphs.
\end{remark}

\section{Definitions and notation}\label{Sdef}

We follow \citet{BCCH16}, with minor variations in the notation.
For the readers convenience, and to set our notation, we repeat the basic
definitions in this section.
See \cite{BCCH16} for further details and references, and see also
\cite{SJ249} for 
further details in the (standard) special
case of probability spaces.

For any topological space $S$, 
 $\cB=\cB(S)$ denotes the Borel $\gs$-algebra on $S$.

$\leb$ denotes the Lebesgue measure 
on $\bbR$.

A \emph{measure space} is, as usual,  a triple $(S,\cF,\mu)$, where $S$ is a
set, 
$\cF$ a \gsa{} on $S$ and $\mu$ a (non-negative) measure on $(S,\cF)$.
We shall often omit $\cF$ and $\mu$ from the notation when they are clear
from the context and denote the measure space just by $S$.
(In contrast to \cite{BCCH16} which is more careful with the notation.)
In particular, we let
$\bbR_+:=\ooo$ denote the measure space $(\bbR_+,\cB,\leb)$, and similarly for
$\oi$ and other intervals  $[0,a]$ and $[0,a)$ with $0< a\le\infty$.

A \emph{subspace} of a measure space $(S,\cF,\mu)$ is a 
measure space $(A,\cF_A,\mu_A)$, where $A$ is a measurable subset of $S$, 
$\cF_A=\set{B\in\cF:B\subseteq A}$ and $\mu_A$ is the restriction of $\mu$
to $\cF_A$.

If $f_1:S_1\to\bbR$ and $f_2:S_2\to\bbR$ are two functions, then $f_1\tensor
f_2:S_1\times S_2\to\bbR$ is the function $f_1\tensor f_2(x,y):=f_1(x)f_2(y)$.

\subsection{Graphons}
A \emph{graphon} $W=(W,S)=(W,S,\cF,\mu)$ is a symmetric integrable function
$W:S\times S\to\bbR$, where $S=(S,\cF,\mu)$ is a \gsf{} measure space.
The space $S$,  its \gsa{} $\cF$ and its measure $\mu$ are important
components of the graphon, but  for convenience we often 
omit them from the notation. (Again, \cite{BCCH16} is more careful.)
We generally identify two graphons that are equal a.e.

Note that in the present paper, as in \cite{BCCH16}, in general, a graphon
is neither required to be bounded nor 
non-negative. 
Note also that
we assume our graphons to be integrable, as  in \cite{BCCH16} (with minor
exceptions, see  \cite[Remarks 2.3 and 2.25]{BCCH16}), while  \cite{VR}
allows for somewhat more general graphons, see 
\cite[Theorem 4.9]{VR}.

We repeat for emphasis that the essential feature of \cite{BCCH16} and the
  present paper is that $\mu$ is allowed to be any \gsf{} measure, and that
  the standard theory in \eg{} 
\cite{BCLSV1}, \cite{Lovasz}, \cite{SJ249}
is the special case when $\mu$ is a probability measure.

A \emph{trivial extension} of a graphon $(W,S,\mu)$ is a graphon
$(\tW,\tS,\tmu)$ 
such that the measure space $(S,\mu)$ is a subspace of $(\tS,\tmu)$ and
\begin{equation}
  \tW(x,y)=
  \begin{cases}
	W(x,y), & x,y\in S,
\\
0, &\text{otherwise}.
  \end{cases}
\end{equation}

\begin{remark}
  We assume, following \cite{BCCH16}, that the measure space where a
  graphon is defined is \gsf. This is mainly because
the standard construction of product measures such as $\qx\mu$ assumes $\mu$
to be \gsf,
since there are serious technical problems otherwise. 
(For example, Fubini's theorem may fail, see \eg{} \cite[Exercise 5.2.1]{Cohn}.)
Nevertheless, it is possible to consider more general measure spaces,
provided we only consider $W$ that vanish outside $S_1\times S_1$ for some
\gsf{} subset $S_1$ (which is reasonable since $W$ should be integrable);
then $W$ is a trivial extension of its restriction to $S_1$. 
We shall not treat this rather trivial extension of the definition
in general and leave it to the reader, but note that
an example of a non-\gsf{} measure space occurs in the proof of \refT{T1} below.
\end{remark}

\subsection{Cut Norm}
If $(S,\cF,\mu)$ is a \gsf{} measure space and $F\in
L^1(\SxS,\allowbreak\mu\times\mu)$, 
then the
\emph{cut norm} of $F$ is defined by
\begin{equation}\label{cn}
  \cn{F}
:= \sup_{T,U}\lrabs{\int_{T\times U} F(x,y)\dd\mu(x)\dd\mu(y)},
\end{equation}
taking the supremum over all measurable $T,U\subseteq S$.
We use also notations such as $\cnx{S}F$ or $\cnx{S,\mu}F$.
Note that 
\begin{equation}
  \label{cutl1}
\cn{F}\le \norm{F}_{L^1(\SxS)}.
\end{equation}
It is easily verified that all properties in 
\cite[Section 4 and Appendix E.1--E.2]{SJ249} 
hold also in the \gsf{} case studied here.
(This includes  other, equivalent,
versions of the cut norm.)
In particular, for any $F\in L^1(\SxS)$,
\begin{equation}\label{cawdor}
  \cn{F}=0\iff F=0 \quad\text{$\qxp{\mu}$-\aex} 
\end{equation}
Moreover,
\begin{equation}\label{macduff}
\cnx{S,\mu}{F}
=
 \lrabs{ \sup_{g,h} \int_{\SxS} F(x,y)g(x)h(y)\dd\mu(x)\dd\mu(y)}
\end{equation}
with the supremum taken over all measurable functions $g,h:S\to\oi$.
As a consequence, for any bounded $f_1,f_2:S\to\bbRp$,
\cf{} \cite[(4.5)]{SJ249},
\begin{equation}
  \label{malcolm}
\cn{f_1(x)f_2(y)F(x,y)}
\le
\normloo{f_1}\normloo{f_2} \cn{F}.
\end{equation}

\subsection{Measure-preserving maps and couplings}
If $\gf$ is a function $S_1\to S_2$, we define for any functions $f$ on $S_2$
and $W$ on $S_2^2$, the \emph{pull-backs}
$f\qphi{}(x):=f(\gf(x))$
and
$W\qphi{}(x,y):=W(\gf(x),\gf(y))$; these are functions on $S_1$ and $S_1^2$,
respectively. 

Similarly, if $\gf:S_1\to S_2$ is measurable, for two measurable spaces
$(S_i,\cF_i)$, and $\mu$ is a measure on $(S_1,\cF_1)$, then the
\emph{push-forward} of $\mu$ is the measure $\mu\qphi{}$ on $(S_2,\cF_2)$
defined 
by $\mu\qphi{}(A):=\mu(\gf\qw(A))$.
Note that $\int_{S_1} f\qphi{}\dd\mu=\int_{S_2}f\dd\mu\qphi{}$ for any
measurable function $f$ on $S_2$ and measure $\mu_1$ on $S_1$ such that one
of the integrals is defined (finite or $+\infty$).
Similarly, if $W\in L^1(\qx{S_2})$, then 
$\int_{\qxq{S_1}}W\qphi{}\dd\mu^2=\int_{\qxq{S_2}}W\dd\qxq{(\mu\qphi{})}$
and
\begin{equation}
  \label{dunsinane}
\cnx{{S_1},\mu}{W\qphi{}}=\cnx{S_2,\mu\qphi{}}{W}.
\end{equation}

A map $\gf:(S_1,\cF_1,\mu_1)\to(S_2,\cF_2,\mu_2)$ is \emph{\mpp} if it is
measurable and $\mu_1\qphi{}=\mu_2$.
Note that all properties in \cite[Section 5]{SJ249} hold also in the \gsf{} case
studied here.  

A \emph{coupling} of two measure spaces $(S_1,\cF_1,\mu_1)$
and $(S_2,\cF_2,\mu_2)$ is a pair $(\gf_1,\gf_2)$ of \mpm{s}
$\gf_i:S\to S_i$ defined on a common measure space $(S,\cF,\mu)$.
We consider in this paper only the \gsf{} case. (Note that $S$ automatically is
\gsf{} if $S_1$ or $S_2$ is.)
An important special case is when $S=S_1\times S_2$ and $\gf_i=\pi_i$, the
projection of $S_1\times S_2$ onto $S_i$, $i=1,2$; we call such couplings
\emph{special}.  In this case $\mu$ is
thus a measure on $S_1\times S_2$ such that $\mu\qpi i=\mu_i$; we call such
a measure $\mu$ a \emph{coupling measure} of $\mu_1$ and $\mu_2$.

If  $(\gf_1,\gf_2)$ is a general coupling of $S_1$ and $S_2$ with
$\gf_i:S\to S_i$, 
then 
$\gf:=(\gf_1,\gf_2)$ is a measurable map $S\to S_1\times S_2$, and
the push-forward measure $\mu\qphi{}$ is a coupling measure of $\mu_1$ and
$\mu_2$. 
Using this, it is easy to see that it suffices to consider special couplings
in, for example, \eqref{dcut1}, \eqref{dl} and \eqref{dp} below.
(In fact, \cite{BCCH16} consider only special couplings.)

Note that a coupling of $S_1$ and $S_2$ exists only if
$\mu_1(S_1)=\mu_2(S_2)$; in that case there always exist coupling measures,
see \cite[Lemma 3.2]{BCCH16}.

\subsection{The cut metric and equivalence}\label{SSdcut}
The cut metric $\dcut(W_1,W_2)$ for two graphons $W_1,W_2$, possibly defined
on different spaces, is defined by \cite{BCCH16} in two steps:
\begin{romenumerate}
\item \label{dcuti}
If $\mu_1(S_1)=\mu_2(S_2)$, then (as in the standard case of
  probability spaces, see \eg{} 
\cite{BCLSV1, Lovasz,SJ249})
\begin{equation}\label{dcut1}
  \dcut(W_1,W_2) :=\inf_{\gfgf}\cn{W_1\qphi1-W_2\qphi2},
\end{equation}
taking the infimum over all couplings $(\gf_1,\gf_2)$ of $S_1$ and $S_2$
(or, as in \cite{BCCH16}, only over special couplings).
\item \label{dcutii}
In general, take trivial extensions $(\tW_i,\tS_i,\tmu_i)$ of
$(W_i,S_i,\mu_i)$ such that 
$\tmu_1(\tS_1)=\tmu_2(\tS_2)$ and define
$\dcut(W_1,W_2):=\dcut(\tW_1,\tW_2)$.
\end{romenumerate}
It is shown in \cite{BCCH16} that this is well-defined, and that the cut metric
satisfies the triangle inequality and thus is a pseudo-metric.

\begin{remark}
   By \eqref{macduff}, for a special coupling with coupling
measure $\mu$
we have explicitly
\begin{multline}\label{yorick}
\cn{W_1\qpi1-W_2\qpi2}
=
\sup_{f,g} 
\biggl|\int_{(S_1\times S_2)^2}
\bigpar{W_1(x_1,y_1)-W_2(x_2,y_2)}\times
\\f(x_1,x_2)g(y_1,y_2)
\dd\mu(x_1,x_2)\dd\mu(y_1,y_2)
\biggr|,
\end{multline}
taking the supremum over measurable $f,g:S_1\times S_2\to\oi$.
\end{remark}

Two graphons $W_1$ and $W_2$ are \emph{equivalent} if $\dcut(W_1,W_2)=0$; in
this case we write $W_1\equ W_2$. (This is sometimes called 'weakly
equivalent'.) 
Since $\dcut$ is a pseudo-metric, $\equ$ is an equivalence relation,
and $\dcut$ is a metric on the set of equivalence classes.
When we talk about metric properties such as completeness and compactness
for $\dcut$,
this should be interpreted as properties in the metric space of equivalence
classes, but for convenience, we usually talk about graphons
rather than equivalence classes.

Note that if $\tW$ is a pull-back $W\qphi{}$ or a trivial extension of a
graphon $W$, then $\tW\equ W$.

We shall repeatedly use the following propositions
shown in \cite{BCCH16}:

\begin{proposition}[{\cite[Proposition 2.8]{BCCH16}}]\label{P2.8}
Every graphon is   equivalent to a graphon defined on the space $\ooo$.
\end{proposition}

\noindent
(Just as a graphon on a
  probability space is equivalent to a graphon on $\oi$, see \eg{} 
\cite[Section 7]{SJ249}.)

\begin{proposition}[{\cite[Proposition 4.3(c)]{BCCH16}}]\label{P4.3c}
If\/ $W_1$ and $W_2$ are graphons defined on $\bbRp$, then
\begin{equation}\label{dcutR}
  \dcut(W_1,W_2) :=\inf_{\gf}\cn{W_1-W_2\qgf},
\end{equation}
taking the infimum over all \mpb{s}  $\gf:\bbRp\to\bbRp$.
\end{proposition}
\noindent
(In other words, in this case, the infimum in \eqref{dcut1} can be restricted
to couplings with $\gf_1$ the identity and $\gf_2$ a bijection.)

\begin{remark}
  \label{Rset}
We sometimes allow ourselves to talk about the set of all graphons,
ignoring the technical
set-theoretical fact that strictly speaking the graphons, as defined in this
paper, form a class and not a set. This can when necessary be circumvented by
the standard method of
restricting the allowed measure spaces $S$ to some sufficiently large set.
In particular, note that by \refP{P2.8}, the equivalence classes of graphons
form a set.
\end{remark}

\subsection{The invariant $L^1$ and $L^p$ metrics $\dl$ and $\dlp$}
The invariant $L^1$-metric $\dl(W_1,W_2)$ is defined by \cite{BCCH16}
in the same way as the cut metric, replacing 
\eqref{dcut1} in Case \ref{dcuti}, \ie{} when $\mu_1(S_1)=\mu_2(S_2)$, by
\begin{equation}\label{dl}
  \dl(W_1,W_2) :=\inf_{\gfgf}\norml{W_1\qphi1-W_2\qphi2},
\end{equation}
and again using trivial extensions as in \ref{dcutii} above for the general
case. 
It is shown in \cite{BCCH16} that this too is well-defined, and a
quasi-metric.
Note that \eqref{cutl1} implies
\begin{equation}\label{dcut<dl}
  \dcut(W_1,W_2)\le\dl(W_1,W_2).
\end{equation}

Moreover, \cite{BCCH16} more generally defines
the invariant $L^p$-metric $\dlp(W_1,W_2)$, where $1\le p<\infty$,
in the same way:
when $\mu_1(S_1)=\mu_2(S_2)$,
\begin{equation}\label{dp}
  \dlp(W_1,W_2) :=\inf_{\gfgf}\norm{W_1\qphi1-W_2\qphi2}_{L^p},
\end{equation}
and in general trivial extensions are used as in \ref{dcutii} above.
However, for $p>1$ we consider only graphons that satisfy
\begin{equation}\label{pcond}
W_i\in L^1(\qx{\mu_i}) \cap L^p(\qx{\mu_i})  
\quad\text{and}\quad
W_i\ge0;
\end{equation}
for such graphons, \cite{BCCH16} shows that 
$\dlp$ is well-defined and a quasi-metric.
\begin{remark}
  To understand the conditions \eqref{pcond},
first recall that we, and \cite{BCCH16}, assume that a
graphon is integrable, \ie, belongs to $L^1$. 
(As said in \refS{Sintro}, \cite{VR} allows somewhat more general graphons,
see also \cite[Remark 2.3]{BCCH16},
but it seems that the cut distance cannot be defined for them.)
Secondly, taking, for example, $W_2=0$, \eqref{dp} yields
$\dlp(W_1,0)=\norm{W_1}_{L^p}$, so we have to assume $W_1,W_2\in L^p$ in
order to have $\dlp(W_1,W_2)$ finite in general; conversely, 
if $W_1,W_2\in L^p$
then  \eqref{dp} yields
$\dlp(W_1,W_2)\le\norm{W_1}_{L^p}+\norm{W_2}_{L^p}<\infty$, so $\dlp$ is finite.
The third condition, $W_i\ge0$, is perhaps more surprising, but it is used in
the proof in \cite{BCCH16} that $\dlp$ is invariant under trivial
extensions, and it is, in fact, necessary for this when $p>1$, see \refE{Edp}.
\end{remark}

\begin{example}[for signed graphons, $\gd_p$ is in general not invariant
	under trivial extensions]
\label{Edp}
Let $W_1=1$ and $W_2=-1$, on the one-point set $\cS=\set{1}$ with measure 
$\mu\set{1}=1$.
Let $\tW_1$ and $\tW_2$ be the trivial extensions to
$\tS=\set{1,2}$, with  
$\tilde\mu\set{1}=\tilde\mu\set{2}=1$. 
Then $\delta_p(W_1,W_2)=\norm{W_1-W_2}_{p}=2$ but, letting
$\gs:\tS\to\tS$ denote the transposition $\gs(1)=2$, $\gs(2)=1$, 
\begin{equation}
\delta_p(\tW_1,\tW_2)
\le\normlp{\tW_1-\tW_2^{\gs}}
=2^{1/p}<\delta_p( W_1, W_2).
\end{equation}
(In fact, equality holds, since there are only two special couplings.)
Hence, 
without the positivity condition in \eqref{pcond},
$\gd_p$ is not preserved by trivial extensions.
\end{example}

As just said, $\dl$ and $\dlp$ (when defined) are quasi-metrics. Moreover,
as will be shown in Theorems \ref{T=} and \ref{T=p},
\begin{equation}\label{viola}
  \dcut(W_1,W_2)=0 \iff
\dl(W_1,W_2)=0 \iff
\dlp(W_1,W_2)=0,
\end{equation}
with the final equivalence assuming that $p>1$ and \eqref{pcond} holds.
Hence, the equivalence $W_1\cong W_2$ is also characterised by  
$\dl(W_1,W_2)=0$, and when \eqref{pcond} holds, by $\dlp(W_1,W_2)=0$.
Consequently $\gd_1$ is a metric on the set of equivalence classes of
graphons, and $\gd_p$ is a metric on the set of equivalence classes of
non-negative graphons in $L^p\cap L^1$. 

Furthermore, by \eqref{dcut<dl}, convergence in $\dl$ implies convergence in
$\dcut$. 
However, this fails for $\dlp$ with $p>1$, see \refE{Epconv}.

\begin{example}[convergence in $\dlp$ with $p>1$
does not imply convergence in cut norm]
\label{Epconv}
  Let $W_n:=n^{-2}\etta_{\qx{[0,n]}}$ on $\bbRp$.
Then, for any $p>1$,
\begin{equation}
\dlp(W_n,0)\le \normlp{W_n}=n^{-2(1-1/p)}\to0 \quad\text{as}\quad \ntoo.
\end{equation}
However,
$\dcut(W_n,0)=\inf_{\gf}\cn{W_n-0\qphi{}}=\cn{W_n}=\norml{W_n}=1$
for every $n$.
Thus convergence in $\dlp$ does not imply convergence in cut norm for any $p>1$.
(This is in contrast to the case of graphons on probability spaces, where
$\dlp\ge\dl\ge\dcut$ because $\normlp\cdot\ge\norml\cdot$.) 
\end{example}

\subsection{Stretched graphons and the stretched metrics $\dcut\sz$,
  $\dl\sz$, $\dlp\sz$}\label{SSstr}

\citet{BCCH16} introduce also a new rescaling of graphons called
\emph{stretching}. 

In general, 
given a graphon $W=(W,S,\mu)$ and $u>0$, we define the rescaled graphon
\begin{equation}\label{str1}
    \stri Wu :=(W,S,u\qq\mu).
\end{equation}
In other words, $\stri Wu$ equals $W$ as a function on $\qxq S$, but we
multiply the underlying measure by $u\qq$.

In the special (and standard) 
case $(S,\mu)=(\bbRp,\leb)$, we can alternatively keep
$(S,\mu)$ and define
\begin{equation}\label{str2}
  \strii Wu(x,y) :=W(u\qqw x,u\qqw y).
\end{equation}
It is easily seen that the two definitions are equivalent up to
equivalence: $\stri Wu\cong\strii Wu$, since
$\strii Wu=(\stri Wu)\qgf$
where $\gf:x\mapsto u\qqw x$ is a \mpm{}
$(\bbRp,\leb)\to(\bbRp,u\qq \leb)$.
(It is the version \eqref{str2} that motivates the name 'stretching'.)
Therefore,  the choice of version usually does not matter, and then we use
the notation 
$\str Wu$ for any of $\stri Wu$ and $\strii Wu$ (when defined).

Note that \eqref{str1} immediately implies
\begin{equation}\label{edmund}
\norml{\str Wu}=u\norml W.  
\end{equation}
As a consequence, again following \cite{BCCH16},
we can normalize any non-zero graphon $W$ to the
\emph{stretched graphon} $W\sz$ defined by
\begin{equation}
  W\sz:=\str{(W)}{\norml{W}\qw}
\end{equation}
with $\norml{W\sz}=1$. For completeness, we also define $W\sz=0$ when $W=0$ a.e.

Furthermore, \cite{BCCH16} define the \emph{stretched metric} $\dcut\sz$ by
\begin{equation}
  \dcut\sz(W_1,W_2):=\dcut(W_1\sz,W_2\sz).
\end{equation}
This is obviously a pseudo-metric on the set of all graphons, and thus a metric
on the corresponding set of equivalence classes; moreover
\begin{equation}\label{sebastian}
  \dcut\sz(W_1,W_2)=0 \iff
W_2\cong\strp {W_1}u \text{\quad for some $u>0$}.
\end{equation}

We can similarly define the pseudo-metrics
$
  \dl\sz(W_1,W_2):=\dl(W_1\sz,W_2\sz)
$ 
and, for $p>1$ and non-negative graphons in $L^p$,
$
  \dlp\sz(W_1,W_2):=\dlp(W_1\sz,W_2\sz).
$ 
As a consequence of \eqref{viola},
\begin{equation}\label{violas}
  \dcut\sz(W_1,W_2)=0 \iff
\dl\sz(W_1,W_2)=0 \iff
\dlp\sz(W_1,W_2)=0,
\end{equation}
with the final equivalence holding for $p>1$ and non-negative graphons in $L^p$.

\section{Graphons and graphs}\label{Sgraphs}

Although the present paper is mainly about graphons, it should be remembered
that the main motivation for studying graphons is the connection to (large)
graphs. 
For the standard case of dense graphs and graphons on a probability space,
see \eg{} \cite{LSz}, \cite{BCLSV1}, \cite{BCLSV2}  and the book
\cite{Lovasz}.

Two aspects of this  connection are treated separately in the following
subsections. 

\subsection{Graphons as limits of graphs}\label{SSWG}

Given a finite graph $G$, the corresponding graphon $W_G$ is defined by
considering the vertex set $V(G)$ as a probability space, with the uniform
measure, and defining $W_G$ on $\qxq{V(G)}$ by $W_G(x,y):=\ett{x\sim y}$.
(This is just the adjacency matrix.)
Alternatively, as is well-known, one can define an equivalent version of
$W_G$ on the standard space
$\oi$ by identifying the vertices of $G$ with disjoint intervals of lengths
$1/|V(G)|$, see \eg{} \cite{Lovasz}, \cite{SJ249}.

In the standard theory \cite{Lovasz}, one says that a sequence of graphs
$G_n$ converges to a graphon $W$ if $|V(G_n)|\to\infty$ and 
\begin{equation}\label{limd}
\dcut(W_{G_n},W)\to0.  
\end{equation}

In the case of sparse graphs, \ie, when the edge density
$|E(G_n)|/|V(G_n)|^2\to0$, \eqref{limd} just gives convergence to the
graphon 0.
In order to get interesting limits, 
\citet{BCCH16} propose instead using stretched
graphons (see \refSS{SSstr} above) and thus the condition
\begin{equation}
  \label{limstr}
\dcut\sz(W_{G_n},W)\to0.
\end{equation}

Given a graph $G$, we can also define a graphon $\xW_G$ by taking the same
function $W_G$ as above (\ie, the adjacency matrix) but consider it as a
graphon defined on the measure space $V(G)$ with the counting measure 
(\ie, each point has measure 1). If $G$ is a finite graph, then $\xW_G$ is a
stretching of $W_G$, see \eqref{str1}, and thus by
\eqref{sebastian}
\begin{equation}\label{malvolio}
  \dcut\sz(W_G,\xW_G)=0.
\end{equation}
Consequently, we can replace $W_{G_n}$ by $\xW_{G_n}$ in
\eqref{limstr}. 
(We can also use any other stretching, for example $W_{G_n}\sz$.)
One technical advantage of $\xW_G$ is that it also is defined for countable
infinite graphs $G$; however, since we want our graphons to be integrable,
we still have to assume that $G$ has only a finite number of edges.

\begin{remark}
 There is also another theory for sparse
  graphs due to \citet{BR09} and further developed by \citet{BCCZ14a}, where
  instead of 
stretching $W_{G_n}$, it is rescaled to $W_{G_n}/\norml{W_{G_n}}$. 
 As discussed in \cite{BCCH16}, 
it seems that the two theories have applications to different types of
sparse graphs. We shall not consider the theory of \cite{BR09} here.
\end{remark}

\subsection{Random graphs defined by graphons}\label{SSRG}
In the standard theory for graphons on a probability space, 
there is a standard definition of a
random graph $G(n,W)$ (with $n$ vertices) for a given graphon $W$ and any
$n\ge1$.
\citet{BCCH16} define a version of this for 
 the present setting as follows.
(Essentially the same construction is given by \citet{VR}.)

Let $W$ be a \oi-valued graphon on a measure space $(S,\mu)$, and assume
that $W$ is not 0 a.e. 
Consider a Poisson point process $\gG$ on $\ooo\times S$, with
intensity $\leb\times\mu$. 
A realization of $\gG$ is a countably infinite set of points
$\set{(t_i,x_i)}$.
Given such a realization, let
$\tG=\tG(W)$ be the infinite graph with 
vertex set $\set{(t_i,x_i)}$, where two vertices
$(t_i,x_i)$ and $(t_j,x_j)$ are connected by an edge with probability
$W(x_i,x_j)$, independently of all other edges (conditionally on $\gG$).

Moreover, let $\tG_t=\tG_t(W)$ be the induced subgraph of $\tG$ consisting
of all 
vertices $(t_i,x_i)$ with $t_i\le t$. (It is useful to think of the
parameter $t_i$ as the time the vertex is born; then $\gG_t$ is the subgraph
existing at time $t$.)

Finally, we let $G_t=G_t(W)$ be the induced subgraph of $\tG_t$ consisting
of all non-isolated vertices.
Note that the vertex set of $\tG_t$ is \as{} infinite for every $t>0$
if $\mu(S)=\infty$,
but the expected number of edges is 
\begin{equation}\label{EE}
  \E |E(\tG_t)|
=   \E |E(G_t)|
=\frac12 t^2\iint W(x,y)\dd\mu(x)\dd\mu(y)<\infty,
\end{equation}
so $G_t$ is \as{} finite for every $t<\infty$.

Note also that the definition defines growing processes $(\tG_t)_{t\ge0}$ and
$(G_t)_{t\ge 0}$ of random graphs. (With $\tG_0=G_0$ empty with no vertices.)

\begin{remark}
  The graphs are  usually regarded as unlabelled, so the identification
(labelling) of the vertices by points in $\ooo\times S$ is mainly for
convenience. Some, equivalent, 
interesting alternative labellings are the following.

\begin{romenumerate}
\item 
Since the measure $\mu$ is \gsf, the coordinates $t_i$ in the point process
$\gG=\set{(t_i,x_i)}$ are \as{} distinct.
Hence, we may just as well use $t_i$ as the label, and let the vertex set of
$\tG$ be $\set{t_i}$. (With the edge probabilities still given by the
$x_i$ as above.) 
The random graph $\tG$ then is exchangeable in the sense that its edge set
is an exchangeable point process on $\bbRp^2$, see further \citet{VR}, where
this property is explored in depth.

\item 
If the measure $\mu$ is atomless, then the coordinates $x_i$ are also \as{}
distinct, so we can use $x_i$ as label and regard the vertex set of the
random graphs defined above as (random) subsets of $S$. The vertex
set of $\tG_t$ then is a Poisson process on $S$ with intensity $t\mu$.
(If $\mu$ has atoms, then this vertex set generally has multiple points that
have to be distinguished.)

\item 
We may use an arbitrary measurable enumeration of the points in $\gG$ or
$\gG_t$ as $\set{(x_i,t_i)}$, and then use $i$ as the label; this means that
the vertex set of $\tG$ is $\bbN$.
The vertex set of $\tG_t$ is $\bbN$ if $\mu(S)=\infty$ and a random finite set
\set{1,\dots,N} when $\mu(S)<\infty$, with  $N\sim\Po(t\mu(S))$.
\end{romenumerate}
\end{remark}

\begin{remark}\label{Rstr}
  Two stretched graphons define the same random graphs up to a change of time.
In fact, if $u>0$ then,
by  the definition above and \eqref{str1},
 the random graphs $\tG_t(\stri Wu)$ are 
constructed using a Poisson point
process $\gG^{(u)}$ on $\ooo\times S$ with intensity $\leb\times u\qq\mu$.
The map $(t,x)\mapsto (u\qq t,x)$ maps this to the Poisson process $\gG$
with intensity $\leb\times\mu$, and thus 
\begin{equation}\label{maria}
  \tG_t(\str Wu) \eqd \tG_{u\qq t}(W),
\end{equation}
in the strong sense that both sides have the same distribution as processes on
\set{t\ge0}. 
\end{remark}

For the limit theory, we consider the corresponding graphons defined in
\refSS{SSWG}. We have $\dcut\sz(W_{G_t},\xW_{G_t})=0$ by \eqref{malvolio}.
Moreover, since $G_t$ is obtained from $\tG_t$ by deleting isolated
vertices,
$\xW_{\tG_t}$ is a trivial extension of $\xW_{G_t}$, and thus
$\dcut(\xW_{G_t},\xW_{\tG_t})=0$, which implies
$\dcut\sz(\xW_{G_t},\xW_{\tG_t})=0$, see \eqref{sebastian}. 
Hence, when using the stretched metric $\dcut\sz$, it does not matter
whether we use $\xW_{\tG_t}$, $\xW_{G_t}$ or $W_{G_t}$.

\citet{BCCH16} prove that the graphs $G_t(W)$ \as{} converge to $W$ in the
strectched metric $\dcut\sz$, \ie, as \ttoo,
\begin{equation}
  \label{orsino}
\dcut\sz(\xW_{\tG_t},W)
=\dcut\sz(\xW_{G_t},W)
=\dcut\sz(W_{G_t},W)
\to0.
\end{equation}

\section{Topological preliminaries}\label{Stop}
Although the definitions and main results are purely measure-theoretic and
do not involve any topology, we shall use some topological notions in some
results and proofs. We use various standard results that can be found
in several references; for convenience we give some specific references to
\cite{Cohn}.

A \emph{Polish} space is a complete separable metric space. 
(Or, more generally, a topological space homeomorphic to such a space.) 

A measurable space is \emph{Borel} 
(also called \emph{standard} \cite{Cohn} or \emph{Lusin} \cite{DM})
if it is isomorphic to a Borel
subset of a Polish space with its Borel $\gs$-field. 
In fact, a Borel measurable space is either isomorphic to $\oi$ (with the
usual Borel $\gs$-field) or countable (with every subset measurable).
A measure space $(\sss,\cF,\mu)$ is
\emph{Borel} if 
$(\sss,\cF)$ is a Borel measurable space; equivalently, if the measure space 
is isomorphic
to a Borel subset of a Polish space equipped with a Borel measure.
See further \cite[Appendix A.2]{SJ249}.

\emph{Compact} and \emph{locally compact} spaces have the standard
definitions. We consider only Hausdorff spaces;
as said above, properties of the cut metric should be interpreted
in the metric space of equivalence classes of
graphons.

\emph{Second countable} also has its standard definition, \ie, that the
topology has a countable basis. Recall that a compact space is second
countable if and only if it is metrizable \cite[Proposition 7.1.12]{Cohn}.

If $K$ is a compact  space, then 
$C(K)$ is
the Banach space of continuous functions $K\to\bbR$.
If $X$ is a locally compact space, then $C_c(X)$ is the space of continuous
functions $f:X\to\bbR$ with compact support 
\begin{equation}
\supp(f):=\overline{\set{x\in	X:f(x)\neq0}}.
\end{equation}
Furthermore, we let $\CI(K)$ and $\CcI(X)$ denote the subsets of functions with
values in \oi.

If $X$ is locally compact and second countable, then a \emph{Radon} measure
on $X$ is a Borel measure $\mu$ such that $\mu(K)<\infty$ for every compact
$K\subseteq X$. A Radon measure $\mu$ defines a positive linear functional
$f\mapsto\int_Xf\dd\mu$ on $C_c(X)$, and by the Riesz representation
theorem, this yields a one-to-one correspondence between Radon measures and
positive linear functionals on $C_c(X)$
\cite[Theorem 7.2.8]{Cohn}. 
(This extends to general locally compact spaces
if one considers only measures that are \emph{regular}
\cite[Section 7.2]{Cohn}; we only need this for second countable spaces,
and then regularity is automatic \cite[Proposition 7.2.3]{Cohn}.)

Note that a locally compact second countable space is $\gs$-compact. Hence a
Radon measure on such a space is \gsf.

We say that a sequence $\mu_n$
of Radon measures on a locally compact second
countable space $X$ converges \emph{vaguely} to a Radon measure $\nu$ if
$\int_X f\dd\mu_n\to\int_X f\dd\nu$ for every $f\in C_c(X)$; this is denoted
by $\mu_n\vto\nu$. (See \cite[Theorem A2.3]{Kallenberg} for some 
properties of the vague topology.)

We shall use the following simple lemma. It is presumably well-known, but we
have not found a reference so for completeness we include a proof.
(We state it for one vaguely convergent sequence. The lemma and its proof
generalize to two vaguely convergent sequences on two, possibly
different, spaces $X$ and $Y$; this says that the product operation is
vaguely continuous for Radon measures on locally compact second countable
spaces.)

\begin{lemma}\label{Lvague}
  Let $X$ be a  locally compact second countable space, 
and let $\mu_n$, $n=1,2,\dots,\infty$,  be Radon
  measures on $X$ such that $\mu_n\vto\muoo$ as \ntoo.
Then $\mu_n\times\mu_n\vto\muoo\times\muoo$ on $X\times X$.
\end{lemma}
\begin{proof}
  Note that since the spaces are second countable, the Borel \gsa{}
  $\cB(X\times X)=\cB(X)\times\cB(X)$.

Let $f\in C_c(X\times X)$.  Then there exists a compact set $K\subseteq X$ such
that $\supp(f)\subseteq K\times K$. Let $K_1$ be a compact subset of $X$ with
$K\subseteq K_1^\circ$, the interior of $K_1$. Then there exists a function
$\psi\in \CcI(X)$ such that $\supp(\psi)\subseteq K_1$
and $\psi(x)=1$ for $x\in K$.

The set of linear combinations $\sum_1^N g_i(x)h_i(y)$ with $g_i,h_i\in
C(K_1)$ is dense in $C(K_1\times K_1)$, for example as a consequence of the
Stone--Weierstrass theorem \cite[V.8.1]{Conway}.
Hence, given any $\eps>0$, there exists such a linear combination
$f_\eps(x,y)=\sum_1^N g_i(x)h_i(y)$ with 
$\sup_{K_1\times  K_1}|f(x,y)-f_\eps(x,y)|<\eps$. 
Let $r_\eps:=f-f_\eps$.
Since $f(x,y)\neq0$ implies $x,y\in K$ and thus $\psi(x)=\psi(y)=1$, it
follows that for $x,y\in K_1$,
\begin{equation}\label{macbeth}
  \begin{split}
  f(x,y)&
=\psi(x)\psi(y)f(x,y) 
= \psi(x)\psi(y)f_\eps(x,y)+\psi(x)\psi(y)r_\eps(x,y)
\\&
=\sum_1^N \tg_i(x)\thx_i(y) +\tr_\eps(x,y),	
  \end{split}
\end{equation}
where $\tg_i(x):=\psi(x)g_i(x)$ and similarly for $\thx_i(x)$ and
$\tr_\eps(x,y)$. 
We extend $\tg_i$ and $\thx_i$ to $X$ by letting them be 0 outside $K_1$;
since $\supp(\psi)\subseteq K_1$, then $\tg_i,\thx_i\in C_c(X)$.
We similarly define $\tr_\eps(x,y)=0$ when $(x,y)\notin K_1\times K_1$.
Then \eqref{macbeth} holds for all $x,y\in X$, with $|\tr_\eps(x,y)|<\eps$.
Since $\mu_n\vto\muoo$, we have, as \ntoo, for each $i\le N$,
{\multlinegap=0pt
\begin{multline*}
\int_{X\times X} \tg_i(x)\thx_i(y)\dd\mu_n(x)\dd\mu_n(y)  
= \int_X\tg_i(x)\dd\mu_n(x)\int_X\thx_i(y)\dd\mu_n(y)
\\
\to
 \int_X\tg_i(x)\dd\muoo(x)\int_X\thx_i(y)\dd\muoo(y)
= \int_{X\times X} \tg_i(x)\thx_i(y)\dd\muoo(x)\dd\muoo(y)  
\end{multline*}}
and hence by \eqref{macbeth},
{\multlinegap=0pt
\begin{multline}\label{duncan}
\int_{X\times X} f(x,y)\dd\mu_n(x)\dd\mu_n(y)  
-\int_{X\times X} f(x,y)\dd\muoo(x)\dd\muoo(y)  
\\
= 
\int_{X\times X} \tr_\eps(x,y)\dd\mu_n(x)\dd\mu_n(y)  
-\int_{X\times X} \tr_\eps(x,y)\dd\muoo(x)\dd\muoo(y)  
+o(1).
\end{multline}}
Furthermore,
choose a function $\psi_1\in C_c(X)$ such that $\psi_1(x)=1$ when $x\in K_1$.
Then $|\tr_\eps(x,y)|\le \eps \psi_1(x)\psi_1(y)$, and thus
\begin{equation}
  \begin{split}
\lrabs{  \int_{X\times X} \tr_\eps(x,y)\dd\mu_n(x)\dd\mu_n(y)}
&\le  \int_{X\times X}\eps \psi_1(x)\psi_1(y)\dd\mu_n(x)\dd\mu_n(y)
\\&
= \eps\biggpar{ \int_{X} \psi_1\dd\mu_n}^2,
\qquad\qquad n\le\infty.
  \end{split}
\end{equation}
Moreover,
$\int_{X} \psi_1\dd\mu_n \to \int_{X} \psi_1\dd\muoo<\infty$,
and thus there exists a constant $M$ (independent of $\eps$) such that
$\int_{X} \psi_1\dd\mu_n \le M$ for all $n\le\infty$.
As a result, if $R_n$ is the \rhs{} of \eqref{duncan}, then
$|R_n|\le 2M^2\eps+o(1)$ and thus $\limsup_{\ntoo} |R_n|\le 2M^2\eps$.
Since $\eps$ is arbitrary, this yields $R_n\to0$. Consequently,
\eqref{duncan} shows that, 
as \ntoo,
\begin{multline}
\int_{X\times X} f \dd(\mu_n\times\mu_n)
=\int_{X\times X} f(x,y)\dd\mu_n(x)\dd\mu_n(y)
\\
\to\int_{X\times X} f(x,y)\dd\muoo(x)\dd\muoo(y)
= \int_{X\times X} f \dd(\muoo\times\muoo).
\end{multline}
Since $f\in C_c(X)$ is arbitrary, this shows
$\mu_n\times\mu_n\vto\muoo\times\muoo$. 
\end{proof}

\begin{lemma}  \label{Lccut}
Let $\mu$ be a finite Borel measure on  a compact metric space $X$. 
Then, for any $F\in L^1(\KxK)$,
\begin{equation}\label{banquo}
\cnx{K,\mu}{F}
=
 \sup_{g,h\in\CI(K)}\lrabs{ \int_{K\times K} F(x,y)g(x)h(y)\dd\mu(x)\dd\mu(y)}.
\end{equation}
\end{lemma}

\begin{proof}
Temporarily denote the \rhs{} of \eqref{banquo} by $\cnc F$.
Then, by \eqref{macduff}, 
$\cnc F\le \cn F$, 
so it suffices to prove the opposite inequality.

Suppose first that $F$ is bounded, say $|F(x,y)|\le M$ for some $M$.
Let $T,U\subset S$ be measurable, and 
let $\eps>0$.
Since $C(K)$ is dense in $L^1(K,\mu)$,
see \cite[Proposition 7.4.2 (and 7.2.3)]{Cohn},
there exists a function $g\in C(K)$ such that $\norml{\etta_T-g}<\eps$.
By replacing $g(x)$ by $\min(1,\max(0,g(x)))$, we may further assume that
$g\in\CI(K)$. Similarly, there exists $h\in\CI(K)$ such that
$\norml{\etta_U-h}<\eps$. 
Then
\begin{equation*}
\bigabs{  F(x,y)\bigpar{\etta_T(x)\etta_U(y)-g(x)h(y)}}
\le M\bigpar{\abs{\etta_T(x)-g(x)}+ \abs{\etta_U(u)-h(y)}}
\end{equation*}
and thus
\begin{equation}
\int_{\KxK}\bigabs{ F(x,y)\bigpar{\etta_T(x)\etta_U(y)-g(x)h(y)}}
\dd\mu(x)\dd\mu(y)
\le 2M\mu(K)\eps.
\end{equation}
Hence,
\begin{equation*}
  \begin{split}
&
\lrabs{\int_{T\times U}F(x,y)
\dd\mu(x)\dd\mu(y)}
=
\lrabs{\int_{\KxK}F(x,y)\etta_T(x)\etta_U(y)
\dd\mu(x)\dd\mu(y)}
\\&\qquad
\le \lrabs{\int_{\KxK}F(x,y)g(x)h(y)
\dd\mu(x)\dd\mu(y)}
+2M\mu(K)\eps
\\&\qquad
\le\cnc{F}	+2M\mu(K)\eps.
  \end{split}
\end{equation*}
Taking the supremum over all $T$ and $U$ and letting $\eps\to0$, we obtain
$\cn F\le \cnc F$, which completes the proof for bounded $F$.

For a general $F$, consider the truncations
\begin{equation}
  F_M(x,y):=F(x,y)\ett{|F(x,y)|\le M}.
\end{equation} 
By the first part of the proof, $\cn {F_M}=\cnc{F_M}$, and thus
\begin{equation*}
  \begin{split}
	\cn F 
&\le \cn{F_M} + \cn{F-F_M}
= \cnc{F_M} + \cn{F-F_M}
\\&
\le
\cnc{F}+\cnc{F-F_M} + \cn{F-F_M}
\le
\cnc{F}+2\norml{F-F_M}.
  \end{split}
\end{equation*}
Furthermore, as $M\to\infty$,
$\norml{F-F_M}\to0$. Consequently, $\cn F\le\cnc F$, which
completes the proof.
\end{proof}

\section{Equivalence}\label{Seq}

We first extend a result by \citet[Lemma 2.6]{BR09} to the present setting
of \gsf{} measure spaces.

\begin{theorem}\label{T1}
Let $(W_1,S_1)$ and $(W_2,S_2)$ be graphons
where $S_i=(S_i,\cF_i,\mu_i)$
are  $\gs$-finite Borel spaces, $i=1,2$.
Then there exist trivial extensions $(\tW_i,\tS_i)$ of\/ $(W_i,S_i)$ and a
coupling $(\gf_1,\gf_2)$ of $\tS_1$ and $\tS_2$ 
such that 
\begin{equation}\label{t1}
\dcut(W_1,W_2)=
 \dcut(\tW_1,\tW_2) =\cn{\tW_1\qphi1-\tW_2\qphi2}.  
\end{equation}
The coupling may be assumed to be special.
\end{theorem}
In other words, the infimum in \eqref{dcut1} is attained for $\tW_1$ and
$\tW_2$. We shall see in \refE{Enot} that in general it is necessary to take
trivial extensions $\tW_1$ and $\tW_2$, even if $\mu_1(S_1)=\mu_2(S_2)$,
unlike the corresponding result
for graphons on probability spaces in \cite{BR09} (see also
\cite[Theorem 6.16]{SJ249}).
Note also that the result is not true for arbitrary measure spaces, not even
in the standard probability space case, see \cite{BCL} for a counter example.

\begin{proof}
First, we note that we may assume that $S_1$ and $S_2$ are atomless.
In general, we let $\hS_i:=S_i\times\oi$ and $\hW_i:=W_i\qpi{i}$, the
pull-back to $\hS_i$. 
The proof below applies to $(\hW_i,\hS_i)$ and shows that there exist trivial
extensions $\chW_i$ to $\chS_i:=\hS_i\cup E_i$
and a coupling $(\gf_1,\gf_2)$ of $\chS_1$ and $\chS_2$ such that
$
\cn{\chW_1\qphi1-\chW_2\qphi2}
=
 \dcut(\hW_1,\hW_2) =
\dcut(W_1,W_2)
$.
Here $\gf_i:S\to\hS_i$ for some measure space $(S,\mu)$,
and $E_i=(E_i,\nu_i)$ are some measure spaces with $E_i$
disjoint from $\hS_i$.
We may assume that $E_i$ also is disjoint from $S_i$ and define $\tS_i:=
S_i\cup E_i$. Let $\tW_i$ be the trivial extension of $W_i$ to $\tS_i$.
Define $\psi_i:\chS_i\to\tS_i$ by $\psi_i=\pi_i$ on $S_i\times\oi$, and
$\psi_i$ the identity on $E_i$.
Then $\psi_i$ is \mpp{}, and $\chW_i=(\tW_i)\qpsi{i}$.
Hence, the pair of mappings $\gf_i^*:=\psi_i\circ\gf_i:S\to\tS_i$ 
give the desired
coupling. 

We may thus assume that $S_1$ and $S_2$ are atomless Borel spaces.
In this case,
by \cite[Lemma 3.1]{BCCH16}, there exists a \mpb{} of $S_i$ onto 
$\bigpar{[0,\mu_i(S_i)),\cB,\leb}$. Hence we may without loss of generality
assume that the measure space
$S_i$ is $\bigpar{[0,m_i),\cB,\leb}$ for some $m_i\in\ooox$.
Moreover, if $m_i<\infty$, we may make a trivial extension of $W_i$ to
$[0,\infty)$. 
Hence we may, and shall, assume that $S_1=S_2=\bbR_+$.

Let $\gd:=\dcut(W_1,W_2)$. 
By the definition \eqref{dcut1}, there exists for every $n$ a coupling measure
$\hmu_n$ on $\bbRp^2$ with both marginals equal to $\leb$ such that 
\begin{equation}\label{ophelia}
  \cnx{\hmu_n}{W_1^{\pi_1}-W_2^{\pi_2}}
<\gd+\xfrac{1}n.
\end{equation}

Let $\ZZ:=\ooox\times\ooox\setminus\set{(\infty,\infty)}$.
Then, $\ZZ$ is an open subset of the compact metric space $\ooox^2$, so
$\ZZ$ is a locally compact second countable space.
Moreover, $\ZZ$ is a Polish space \cite[Proposition 8.1.1]{Cohn}
and $\bbRp^2\subset\ZZ\subset\ooox^2$.
The measure $\hmu_n$ is defined on $\bbRp^2$, and we can regard it as a
measure on $\ZZ$.

Let, for $N\in\bbN$,  $K_N:=([0,N]\times\ooox) \cup (\ooox\times[0,N])$.
Then each $K_N$ is a compact subset of $\ZZ$ and
$\ZZ=\bigcup_{N\ge1}K_N$; 
moreover, every compact
subset of $\ZZ$ is a subset of some $K_N$.

For each $n$, since $\hmu_n$ has marginals $\leb$,
\begin{equation}\label{hamlet}
  \begin{split}
  \hmu_n(K_N)
&\le \hmu_n([0,N]\times\ooox)+\hmu_n (\ooox\times[0,N])
\\&
=\leb([0,N])+\leb([0,N])
=2N.	
  \end{split}
\end{equation}
Hence, $\hmu_n(K)<\infty$ for every compact $K\subset\ZZ$, so $\hmu_n$ is a
Radon measure on $\ZZ$.
Moreover, \eqref{hamlet} implies that
the sequence $\hmu_n(K)$ is bounded for each compact $K\subset\ZZ$,
which means that the sequence $\hmu_n$ is relatively compact in the vague
topology, see \cite[Theorem A2.3(ii)]{Kallenberg}.
Furthermore, 
by \cite[Theorem A2.3(i)]{Kallenberg}, 
the set of Radon measures on $\ZZ$ with the vague  topology is
metrizable. (In fact, a Polish space.) 
Consequently, there exists a 
subsequence $(\hmux_n)$
of $(\hmu_n)$ that converges vaguely to some Radon measure $\nu$ on $\ZZ$, \ie,
\begin{equation}
  \label{nu}
\lim_\ntoo
\int_{\ZZ} f\dd \hmux_n
= \int_{\ZZ} f\dd \nu, 
\qquad f\in C_c(\ZZ).
\end{equation}
Since $\ZZ\subset\ooox^2$, we may also regard $\nu$ as a measure on
$\ooox^2$. 
Since $\ZZ$ is $\gs$-compact,
$\nu$ is \gsf{} on $\ZZ$ and thus on $\ooox^2$.
(But note that $\nu$ is an infinite measure and thus not a Radon measure on
the compact space $\ooox^2$.)

We next consider the marginals $\nu\qpi{i}$ of $\nu$; these are measures on
$\ooox$. 
If $g\in C_c(\bbRp)$, then the function $g\qpi1(x,y)=g(x)$ has compact
support $\supp(g)\times\ooox$ in $\ZZ$, so $g\qpi1\in C_c(\ZZ)$ and
\eqref{nu} yields, recalling that each $\hmux_n$ has marginals 
$(\hmux_n)\qpi1=\leb$,
\begin{equation}
\int_{\bbRp}g\dd\nu\qpi1=
\int_{\ZZ} g\qpi1\dd \nu
=
\lim_\ntoo \int_{\ZZ} g\qpi1\dd \hmux_n
=
\lim_\ntoo \int_{\bbRp} g\dd (\hmux_n)\qpi1
=\int_{\bbRp} g\dd\leb.
\end{equation}
Consequently, the marginal $\nu\qpi1$ of $\nu$ equals $\leb$ on $\bbRp=\ooo$.
By symmetry, the same holds for $\nu\qpi2$.
However, note that each marginal also may have a point mass at \set\infty;
this point mass may even be infinite, in which case the marginal is not \gsf.
(We shall see that this causes no serious problem.)

Let $\eps>0$.
Since 
$W_i\in L^1(\bbRp^2)$ and
$C_c(\bbRp^2)$ is dense in $L^1(\bbRp^2)$, there exist $\bW_i\in
C_c(\bbRp^2)$ such that $\norm{W_i-\bW_i}_{L^1(\bbRp^2)}<\eps$, $i=1,2$.
Then
\begin{equation}\label{horatio}
  \cnx{\hmux_n}{W_i\qpi{i}-\bW_i\qpi i}
\le
  \norm{W_i\qpi{i}-\bW_i\qpi i}_{L^1(\hmux_n^2)}
=
  \norm{W_i-\bW_i}_{L^1(\leb^2)}<\eps,
\end{equation}
and thus by \eqref{ophelia} and the triangle inequality, recalling that
$\hmux_n=\hmu_{k_n}$ for some $k_n\ge n$,
\begin{equation}\label{laertes}
  \cnx{\hmux_n}{\bW_1^{\pi_1}-\bW_2^{\pi_2}}
<\gd+\xfrac1n+2\eps.
\end{equation}

Extend each $W_i$ and $\bW_i$  trivially (\ie, by 0) to $\ooox^2$.

Let $N$ be such that $\supp(\bW_i)\subseteq [0,N]\times[0,N]$ for $i=1,2$.
Then $\supp(\bW_i\qpi i)\subseteq K_N\times K_N$, and thus
$\bW_i\qpi i\in C_c(\ZZ^2)$.

Let $f,g\in\CI(K_N)$. We can extend $f$ and $g$ to functions in $\CcI(\ZZ)$;
moreover, there exist sequences $f_m$ and $g_m$ in $\CcI(\ZZ)$ such that
\begin{equation}\label{rosencrantz}
  f_m(x_1,x_2)\to
  \begin{cases}
	f(x_1,x_2), & (x_1,x_2)\in K_N,\\
0,&\text{otherwise,}
  \end{cases}
\end{equation}
as \mtoo, and similarly for $g_m$.
Then, by \eqref{laertes} and \eqref{yorick}, for any $m$ and $n$,
\begin{equation}\label{polonius}
  \begin{split}
&  \lrabs{\int_{\ZZ^2}\bigpar{\bW_1(x_1,y_1)-\bW_2(x_2,y_2)}f_m(x_1,x_2)g_m(y_1,y_2)
\dd\hmux_n(x_1,x_2)\dd\hmux_n(y_1,y_2)}
\\&
\qquad\le \gd+2\eps+\xfrac1n.	
  \end{split}
\raisetag\baselineskip
\end{equation}
The integrand in \eqref{polonius} is a continuous function with compact
support in $\ZZ^2$, and by \refL{Lvague},
$\hmux_n\times\hmux_n\vto\nu\times\nu$.  
Hence, we can take the limit as \ntoo{} in \eqref{polonius} and obtain,
with $z_i=(x_i,y_i)$,
\begin{equation}\label{gertrude}
  \begin{split}
& \lrabs{\int_{\ZZ^2}\bigpar{\bW_1\qpi1(z_1,z_2)-\bW_2\qpi2(z_1,z_2)}
  f_m(z_1)g_m(z_2)
\dd\nu(z_1)\dd\nu(z_2)}
\le \gd+2\eps.
  \end{split}
\end{equation}
Now let \mtoo; by \eqref{rosencrantz} and dominated convergence (noting that
$\bW_1\qpi1-\bW_2\qpi2\in L^1(\qx\nu)$),
the integral in \eqref{gertrude} 
converges to
\begin{equation}
\int_{\qx{K_N}}\bigpar{\bW_1\qpi1(z_1,z_2)-\bW_2\qpi2(z_1,z_2)}
  f(z_1)g(z_2)
\dd\nu(z_1)\dd\nu(z_2).
\end{equation}
Since  $f,g\in\CI(K_N)$ are arbitrary, \eqref{gertrude} and \refL{Lccut} 
thus yield
\begin{equation}
  \cnx{K_N,\nu}{\bW_1\qpi 1-\bW_2\qpi2}\le \gd+2\eps.
\end{equation}
Furthermore, $\bW_1\qpi1$ and $\bW_2\qpi2$ vanish off $\qx{K_N}$, and thus
\begin{equation}\label{fortinbras}
  \cnx{\ooox^2,\nu}{\bW_1\qpi 1-\bW_2\qpi2}
=  \cnx{K_N,\nu}{\bW_1\qpi 1-\bW_2\qpi2}\le \gd+2\eps.
\end{equation}
Consequently, 
on $\ooox^2$,
using the analogue of \eqref{horatio} for $\nu$,
\begin{equation}
  \begin{split}
  \cnx{\nu}{W_1\qpi 1-W_2\qpi2}
&
\le  \cnx{\nu}{\bW_1\qpi 1-\bW_2\qpi2}
+  \cnx{\nu}{W_1\qpi1-\bW_1\qpi 1}
+  \cnx{\nu}{W_2\qpi2-\bW_2\qpi 2}
\\&
\le \gd+4\eps.	
  \end{split}
\raisetag\baselineskip
\end{equation}
Since $\eps>0$ is arbitrary, this yields
\begin{equation}\label{guildenstern}
  \cnx{\ooox^2,\nu}{W_1\qpi 1-W_2\qpi2}\le \gd
=\dcut(W_1,W_2).
\end{equation}
If $\nu(\set\infty\times\ooo)$ and $\nu(\ooo\times\set\infty)$ are finite,
then the projections $\nu\qpi1$ and $\nu\qpi2$ are \gsf{} measures on
$\ooox$.
In this case, \eqref{guildenstern} shows that if we take the trivial
extensions of $W_1$ and $W_2$ to $\ooox^2$, then 
$\nu$ is a coupling measure such that the special coupling $(\pi_1,\pi_2)$
yields equality in \eqref{t1}. 
(Recall that 
$\dcut(W_1,W_2)\le  \cnx{\nu}{W_1\qpi 1-W_2\qpi2}$ 
by the definition \eqref{dcut1}.)

In general, the projections are not \gsf, since they may have infinite atoms
at $\infty$,
but we may easily modify the
construction. Let $\tS:=\bbRp\cup\bbRp'$, where $\bbRp$ and $\bbRp'$ are two
disjoint copies of $\ooo$, with variables denoted $x$ and $x'$,
respectively.
Define a map $\gf:\ZZ\to\qx{\tS}$ by
\begin{equation}\label{benedict}
  \gf(x,y)=
  \begin{cases}
(x,y), & x,y<\infty,
\\
(x,x'),& y=\infty,
\\(y',y),&x=\infty.
  \end{cases}
\end{equation}
This means that the lines $\bbRp\times\set\infty$ and
$\set\infty\times\bbRp$ are mapped to the diagonals in $\bbRp\times\bbRp'$
and $\bbRp'\times\bbRp$, respectively.
 Let $\tnu:=\nu\qphi{}$. Then there is no problem with the projections:
 $\tnu:=\nu\qphi{}$ 
is a \gsf{} measure on $\qx{\tS}$, and the marginals 
$\tnu_i:=\tnu\qpi i$ are \gsf{} measures on $\tS$; moreover, $\tnu_i=\leb$ on
$\bbRp\subset \tS$.
Hence, we can define $\tS_i:=(\tS,\tnu_i)$ and let $\tW_i$ be the trivial
extension of $W_i$ to $\tS_i$; then $\tnu$ is a coupling measure.
Moreover,
\begin{equation}\label{bottom}
  \cnx{\qx{\tS},\tnu}{\tW_1\qpi 1-\tW_2\qpi2}
=  \cnx{\ooox^2,\nu}{W_1\qpi 1-W_2\qpi2}.
\end{equation}
One way to see \eqref{bottom} is to define $\psi:\tS\to\ooox$ by $\psi(x)=x$
and $\psi(x')=\infty$, and let $\psiq:=\psi\tensor\psi:\tS^2\to\ooox^2$.
Then it follows from \eqref{benedict} that $\psiq\circ\gf$ is the identity
embedding $Z\to\ooox^2$.
Hence, $\psiq$ is \mpp{} $(\qx{\tS},\tnu)\to(\qx{\ooox},\nu)$.
Furthermore, $\tW_i=W_i^\psi$ and thus
$\tW_i\qpi{i}=(W_i^\psi)\qpi{i}=W_i^{\psi\circ\pi_i}=W_i^{\pi_i\circ\psiq}$,
and thus
\begin{equation*}
  \cnx{\qx\tS,\tnu}{\tW_1\qpi 1-\tW_2\qpi2}
=
  \cnx{\qx\tS,\tnu}{\tW_1^{\pi_1\circ\psiq}-\tW_2^{\pi_2\circ\psiq}}
=  \cnx{\ooox^2,\nu}{W_1\qpi 1-W_2\qpi2},
\end{equation*}
showing \eqref{bottom}.

Finally, \eqref{bottom} and  \eqref{guildenstern} show that
\begin{equation}
  \cnx{\qx\tS,\tnu}{\tW_1\qpi 1-\tW_2\qpi2}
=  \cnx{\ooox^2,\nu}{W_1\qpi 1-W_2\qpi2}\le 
\dcut(W_1,W_2),
\end{equation}
and thus equality holds by the definition \eqref{dcut1}.
\end{proof}

\begin{remark}\label{RT1dl}
  The analogue of \refT{T1} for $\dl$ holds too. The proof is essentially
  the same, with the difference that we do not need $f,g,f_m,g_m$; we
  proceed directly from the $L^1$ version of \eqref{laertes} to the $L^1$
  version of \eqref{fortinbras} using \refL{Lvague} and the fact that
  $|\bW_1\qpi1-\bW_2\qpi2|\in C_c(\ZZ^2)$.

Similarly, \refT{T1} holds for $\dlp$ too, for any $p>1$, provided
$W_1,W_2\in L^p$ and $W_1,W_2\ge0$.
\end{remark}

As a special case, we obtain the following characterisation of equivalent
graphons on  Borel spaces, for example $\bbRp$;
again this extends a result by \citet[Corollary 2.7]{BR09}.
\begin{theorem}\label{T=0}
Let $(W_1,S_1)$ and $(W_2,S_2)$ be graphons
with $\dcut(W_1,W_2)=0$, and assume that $S_i=(S_i,\cF_i,\mu_i)$
are  $\gs$-finite Borel spaces, $i=1,2$.
Then there exist trivial extensions $(\tW_i,\tS_i)$ of\/ $(W_i,S_i)$ and a
coupling $(\gf_1,\gf_2)$ of $\tS_1$ and $\tS_2$ 
such that $\tW_1^{\gf_1} = \tW_2^{\gf_2}$ a.e.
The coupling may be assumed to be special.
\end{theorem}
\begin{proof}
  An immediate consequence of \refT{T1} and \eqref{cawdor}.
\end{proof}

\begin{remark}
  In the standard case of graphons on probability spaces, \citet{BCL} proved
  a similar result using pull-backs in the opposite direction, \viz{}
  that
$W_1\cong W_2$ if and only if there exists a graphon $W$ on some probability
  space $S$ and mesure preserving maps $\gf_i:S_i\to S$ such that
  $W_i=W\qphi{i}$ a.e.
See also \cite[Theorem 8.4]{SJ249} with an alternative proof.
(For this result, no assumption on the probability spaces $S_1$ and $S_2$ is
  needed, so it yields a general characterisation of equivalence.)

We conjecture that this result too extends (in some form) to the \gsf{}
case, but we leave this as an open problem.
\end{remark}

We can elaborate \refT{T=0} as follows, \cf{} the corresponding result
\cite[Theorem 8.3]{SJ249} for graphons on probability spaces.

\begin{definition}\label{Dee}
  Two graphons $(W_1,S_1)$ and $(W_2,S_2)$ are \emph{\eeq}
if one of the following conditions holds:
\begin{romenumerate}
\item \label{Deepb}
$W_1=W_2\qgf{}$ \aex{} for some \mpm{} $\gf:S_1\to S_2$, or conversely.
\item \label{Deetr}
$W_1$ is a trivial extension of $W_2$, or conversely.
\end{romenumerate}
\end{definition}

Note that \ref{Deepb} includes the case of two \aex{} equal graphons on the
same measure space.

\begin{theorem}\label{T=}
  Let $W$ and $W'$ be two graphons. Then the following are equivalent:
  \begin{romenumerate}
  \item \label{T=equ}
$W\equ W'$.
  \item \label{T=d}
$\dcut(W,W')=0$.
  \item \label{T=1}
$\dl(W,W')=0$.
  \item \label{T=e}
There exists a finite sequence of graphons $W=W_0,\dots,W_N=W'$ such that
$W_{i-1}$ and $W_i$ are \eeq{} for $i=1,\dots,N$.
  \end{romenumerate}
\end{theorem}

The equivalence \ref{T=d}$\iff$\ref{T=1} is \cite[Proposition 2.6]{BCCH16}.

\begin{proof}
  \ref{T=equ}$\iff$\ref{T=d} holds by definition.

  \ref{T=d}$\implies$\ref{T=e}: 
By \cite[Lemma 4.1]{BCCH16}, $W=W_1^{\gf_1}$ and $W'=W_2^{\gf_2}$ for some
graphons $W_1$ and $W_2$ on Borel measure spaces $S_1$ and $S_2$.
Then $\dcut(W_1,W_2)=\dcut(W,W')=0$.
Hence \refT{T=0} shows the existence of trivial extensions $\tW_i$ of
$W_i$ and a coupling $(\psi_1,\psi_2)$ such that
$\tW_1\qpsi1=\tW_2\qpsi2$ a.e.
Consequently,
$W, W_1,  \tW_1, \tW_1\qpsi1,
\tW_2\qpsi2, \tW_2, W_2, W'$
is a sequence where every graphon is \eeq{} to the next.

  \ref{T=e}$\implies$\ref{T=1}: It is clear by the definitions in
  \refS{Sdef} that $\dl(W,W')=0$ whenever $W$ and $W'$ satisfy one of the
	two cases in \refD{Dee}. The result follows by the triangle inequality.

  \ref{T=1}$\implies$\ref{T=d}: Immediate by \eqref{dcut<dl}.
\end{proof}

\begin{theorem}
  \label{T=p}
Let $p>1$ and let $W$ and $W'$ be two non-negative graphons in $L^p$.
Then the conditions \ref{T=equ}--\ref{T=e} in \refT{T=} are also equivalent to
\begin{romenumerateq}
\item \label{T=dp}
$\dlp(W,W')=0$.
\end{romenumerateq}
\end{theorem}
\begin{proof}
  \ref{T=e}$\implies$\ref{T=dp}: Immediate, as the proof of
  \ref{T=e}$\implies$\ref{T=1} above.

  \ref{T=dp}$\implies$\ref{T=e}: As the proof of
  \ref{T=d}$\implies$\ref{T=e} above, now using \refR{RT1dl}.

Alternatively, we can prove   \ref{T=dp}$\implies$\ref{T=1} directly
as follows.
Using \refP{P2.8}, 
we may assume that $W_1$ and $W_2$ are defined on $\bbRp$.

Let $\eps>0$.  Since $\dlp(W,W')=0$, there exists 
by \cite[Proposition 4.3(c) and Remark 4.4]{BCCH16}
a \mpb{}
$\gf:\bbRp\to\bbRp$ such that 
$\normlp{W-(W')\qphi{}}<\eps$.

Also, let $M>0$. Let $I_M:=[0,M]$ and $A:=I_M\cup\gf\qw(I_M)$.
Then, $\leb(A)\le\leb(I_M)+\leb(\gf\qw(I_M))=2M$ and thus, by \Holder's
inequality,
\begin{equation}\label{oberon}
  \begin{split}
\norml{(W-(W')\qgf{})\etta_{\qx A}}
\le \normlp{W-(W')\qgf{}}(\lebb(A\times A))^{1-1/p}
<\eps(2M)^{2-2/p}.
  \end{split}
\end{equation}
Moreover, 
\begin{equation}\label{titania}
  \begin{split}
\norml{W\etta_{(\qx A)^c}}
\le
\norml{W\etta_{(\qx {I_M})^c}}
  \end{split}
\end{equation}
and
\begin{equation}\label{puck}
  \begin{split}
\norml{(W')\qphi{}\etta_{(\qx A)^c}}
&\le
\norml{(W')\qphi{}\etta_{(\qx{\gf\qw(I_M)})^c}}
=
\norml{(W'\etta_{(\qx{I_M})^c})\qphi{}}
\\&
=
\norml{W'\etta_{(\qx{I_M})^c}}.
  \end{split}
\raisetag\baselineskip
\end{equation}
Consequently, by the triangle inequality and \eqref{oberon}--\eqref{puck},
\begin{equation}
  \begin{split}
\gd_1(W,W')
&\le
\norml{W-(W')\qgf{}}
\\&
<\eps(2M)^{2-2/p}
  +
\norml{W\etta_{(\qx{I_M})^c}}	
+\norml{W'\etta_{(\qx{I_M})^c}}.	
  \end{split}
\end{equation}
Letting first $\eps\to0$ and then $M\to\infty$, we obtain $\dl(W,W')=0$.
\end{proof}


\refT{T=} has an important consequence. In order to state it generally,
let a \emph{property} of a graphon be anything that is determined uniquely
by the graphon; \ie, any function $\Phi$ from the set of all
graphons to some arbitrary set $X$.

\begin{theorem}
  \label{TC=}
Let $\Phi$ be a property of graphons such that $\Phi(W_1)=\Phi(W_2)$
whenever $W_1$ and $W_2$ are \eeq{} graphons.
Then  $\Phi(W)=\Phi(W')$ whenever 
$W$ and $W'$ are two equivalent graphons, and consequently, $\Phi$ is
well-defined on the set of equivalence classes of graphons.
\end{theorem}
\begin{proof}
  Obvious by \refT{T=}\ref{T=equ}$\implies$\ref{T=e}.
\end{proof}

Note that it does not matter whether the property depends continuously on
the graphon or not.

\begin{example}\label{E=}
We give some examples of applications of \refT{TC=}. In all of them, it is
immediately verified that two \eeq{} graphons give the same result.
In the first examples, the property $\Phi$ is  binary  (true or false); 
the last examples are real
parameters. Another example, with a more complicated $\Phi$, follows in the
proof of \refT{T2}.
\begin{romenumerate}
\item \label{E=>0}
If the graphon $W$ is \aex{} non-negative, 
then so is every graphon equivalent to $W$.

\item 
If the graphon $W$ is \aex{}  $\oi$-valued,
then so is every graphon equivalent to $W$.

\item \label{E=Lp}
For any $p>0$, the 
$L^p$-norm $\normlp{W}\in\ooox$ is the same for any two equivalent graphons,
  and is 
  thus well-defined for every equivalence class of graphons.

\item 
Let $h(p):=-p\log p-(1-p)\log(1-p)$ for $p\in\oi$, and define 
  the \emph{entropy} of a $\oi$-valued graphon $W$ by
  \begin{equation}\label{entropy}
\cE(W):=\int_{\qx S} h(W(x,y))\dd\mu(x)\dd\mu(y)\in\ooox 	.
  \end{equation}
Then $\cE(W)$
is the same for any two equivalent $\oi$-valued graphons,
  and is 
  thus well-defined for every equivalence class of such graphons.

For the entropy \eqref{entropy} in the standard case of graphons on
probability 
spaces, see \eg{} \cite{Aldous85}, 
\cite[Appendix D.2]{SJ249}, 
\cite{SJ288} and the further references given there.
We leave it as an open problem whether  results in these references can be
extended to the present
setting, at least under suitable conditions.
\end{romenumerate}
\end{example}

As another application, we obtain a new (simpler) proof of 
\cite[Theorem  2.21]{BCCH16}. (The statement is slightly expanded here.)

\begin{theorem}[\cite{BCCH16}]\label{T2}
  Let $W_1$ and $W_2$ be two graphons. Then the following are equivalent.
  \begin{romenumerate}
  \item\label{T2=}
 $W_1\cong W_2$.
  \item\label{T2tG} 
The random graph processes $(\tG_t(W_1))_{t\ge0}$
and $(\tG_t(W_2))_{t\ge0}$ have the same distribution, up to vertices that
stay isolated for all $t$.
  \item\label{T2G} 
The random graph processes $(G_t(W_1))_{t\ge0}$
and $(G_t(W_2))_{t\ge0}$ have the same distribution.
  \item\label{T2tGt} 
The random graphs $\tG_t(W_1)$
and $\tG_t(W_2)$ have the same distribution, up to isolated vertices, for
every $t>0$. 
  \item\label{T2Gt} 
The random graphs $G_t(W_1)$
and $G_t(W_2)$ have the same distribution, for
every $t>0$. 
  \end{romenumerate}
\end{theorem}

Note that a trivial extension of a graphon in general adds permanently
isolated vertices 
to $\tG_t$, so the result would not be true without the provisions for them in
\ref{T2tG} and \ref{T2tGt}.

\begin{proof}
  \ref{T2=}$\implies$\ref{T2tG}.
By \refT{TC=} (taking $\Phi(W)$ to be the distribution of the process
$(\tG_t(W))_t$), it suffices to consider the case of two \eeq{} graphons,
which  is obvious for both cases in \refD{Dee}.

  \ref{T2tG}$\implies$\ref{T2G}$\implies$\ref{T2Gt}
and
  \ref{T2tG}$\implies$\ref{T2tGt}$\implies$\ref{T2Gt} are trivial.

  \ref{T2Gt}$\implies$\ref{T2=}.
We can couple the random graphs for each fixed $t$ such that
$G_t(W_1)=G_t(W_2)$ a.s., and thus
$W_{G_t(W_1)}=W_{G_t(W_2)}$ a.s.
Moreover,  \eqref{orsino} shows that
$\dcut\sz\bigpar{W_{G_t(W_i)}, W_i}\to0$ in probability as \ttoo, for
$i=1,2$.
Consequently, 
\begin{equation}
  \dcut\sz(W_1,W_2)=0.
\end{equation}
This implies by \eqref{sebastian} that $W_2\cong\strp {W_1}u$ for some
$u>0$.
Using \eqref{maria} and the assumption \ref{T2Gt}, this in turn implies
\begin{equation}
G_t(W_1)\eqd
  G_t(W_2)
\eqd   G_t(\strp{W_1}u)
\eqd  G_{u\qq t}(W_1).
\end{equation}
However, this equality implies (except in the trivial case $W_1=0$) that
$u=1$, for example by considering the expected number of edges, see \eqref{EE}.
Hence $W_2\cong\strp{W_1}1=W_1$.
\end{proof}

\begin{example}
[trivial extensions are necessary in Theorems \ref{T1} and	\ref{T=0}] 
\label{Enot}
Define $W(x,y):=e^{-x-y}\ett{x\ge0}\ett{y\ge0}$ for real $x,y$.
Let $a\in(0,\infty]$ and let $S_1:=[-a,\infty)$, $S_2:=[0,\infty)$, both with
Lebesgue measure. Define the graphons $(W_1,S_1)$ and $(W_2,S_2)$ by
$W_1=W$ and $W_2=W$. Then $W_1$ is a trivial extension of $W_2$, and
thus $\dcut(W_1,W_2)=0$. Nevertheless, there exists no coupling 
$(\gf_1,\gf_2):S\to(S_1,S_2)$ 
such that $\cn{W_1\qphi1-W_2\qphi2}=0$
and thus, by \eqref{cawdor}, $W_1\qphi1=W_2\qphi2$ a.e. 
In fact, suppose that such
a coupling exists, and let $D_i(x):=\int_{S_i}W_i(x,y)\dd\leb(y)$, $x\in S_i$.
Then, for \aex{} $x\in S$,
since $\mu\qphi i=\leb$,
\begin{equation}  
  \begin{split}
  D_1\qphi1(x)
&
=\int_{S_1}W_1(\gf_1(x),y)\dd\leb(y)
=\int_{S}W_1(\gf_1(x),\gf_1(z))\dd\mu(z)
\\&
=\int_{S}W_1\qphi1(x,z)\dd\mu(z)
=\int_{S}W_2\qphi2(x,z)\dd\mu(z)
=D_2\qphi2(x),
  \end{split}
\end{equation}
where the final equality follows by symmetry.
However, 
\begin{equation}
  \mu\set{x\in S:  D_1\qphi1(x)=0}
= \leb\set{y\in S_1:  D_1(y)=0}=a,
\end{equation}
while, similarly,
$  \mu\set{x\in S:  D_2\qphi2(x)=0}=0$.
This is a contradiction.
\end{example}

\section{Completeness}\label{Scomplete}

\citet[Corollary 2.13]{BCCH16} show that the set of (equivalence classes of)
\oi-valued graphons is complete for the
$\dcut$ metric.
Their proof is based on their characterisation of compactness for (\oi-valued)
graphons, see \refS{Scompact}.
We discuss completeness further in this swection, and give a new, more
direct, proof of 
their result; we also extend the result somewhat. (Whether our proof is
simpler or not is a matter of taste and background.)
We also investigate in more detail why a
restriction such as $\oi$-valued is needed; the set of all graphons is not
complete and we give several examples that illustrate that.

\subsection{Completeness for the cut norm}
We begin by studying convergence in the cut norm $\cn\cdot$
for graphons on a fixed \gsf{}
measure space.

We say that a set $\cA$ of integrable functions (possibly defined on different
measure spaces) is \emph{uniformly integrable} if it
satisfies
\begin{PXenumerate}{UI}
\item\label{wc1} $\sup_{F\in\cA}\int_S|F|\dd\mu<\infty$,
and
\item\label{wc2} 
$\sup_{F\in\cA}\int_{|F|>B}|F|\dd\mu\to0$ as 
$B\to\infty$.
\end{PXenumerate}
This is a standard notion for probability spaces, where furthermore \ref{wc2}
implies \ref{wc1}. (For infinite measure
spaces, the notion is much less used and when used, the
definitions vary between different authors.) 
Note also  that assuming \ref{wc1}, \ref{wc2} is equivalent to
\begin{PQenumerate}{UI2$'$}
\item \label{wc2'}
$\sup_{E\subseteq S:\mu(E)\le\gd}\sup_{F\in\cA}\int_{E}|F|\dd\mu\to0$ as 
$\gd\to0$.
\end{PQenumerate}

We say that a set $\cA$ is \emph{\sui} if it satisfies \ref{wc2}.
(For a finite measure space, \sui{} is thus equivalent to \ui, but in
general it is weaker.)

Recall that a subset of a Banach space is relatively weakly compact 
if it is a subset of a weakly compact set.
Recall also that a set is relatively  weakly compact if and only if it is
sequentially weakly compact, \ie, every sequence in the set has a convergent
subsequence. 
(The Eberlein--{\v S}mulian theorem \cite[Theorem V.6.1]{DS}.)
Moreover, let $(S,\mu)$ be a \gsf{} measure space, and let $S_n$ be an
increasing sequence of subsets of $S$ with finite measures such that
$S=\bigcup_n S_n$.
Then a subset
$\cA$ of $L^1(S,\mu)$
is relatively weakly compact if and only if 
\ref{wc1}--\ref{wc2} hold together with
\begin{PQenumerate}{WC3}
\item\label{wc3} $\sup_{F\in\cA}\int_{S\setminus S_n}|F|\dd\mu\to0$ as \ntoo.
\end{PQenumerate}
This is (one form of)
the Dunford--Pettis theorem; see \cite[Theorem 3.2.6]{DP} and 
\cite[Theorem IV.8.9]{DS} for two slightly different formulations, and note
\cite[Corollary IV.8.10]{DS}.

The Dunford--Pettis theorem is perhaps best known in the case of a
probability space. In that case, and more generally for any finite measure
space $(S,\mu)$, \ref{wc3} is trivial;
thus,
the theorem then says that a subset of $\lsmu$ is relatively weakly compact
if and only it 
is \ui.

We say that a set $\cA$ in a metric space is \emph{relatively complete} if
every Cauchy sequence in $\cA$ converges to some limit (which does not have to
belong to $\cA$). It is easy to see that $\cA$ is relatively complete if and
only if $\overline\cA$ is complete.

We now give our main result on completeness for the cut norm.
Although we only are interested
in symmetric functions, the theorem and its proof hold for general
(integrable) functions on a product $S_1\times S_2$; such functions appear
for example in the study of bipartite graphs.

\begin{theorem}
 \label{TXA}
Let $(S,\mu)$ be a \gsf{} measure space.
Any \ui{} set of  graphons in $L^1(S^2,\mu^2)$ is
relatively   complete for the cut norm.
\end{theorem}

\begin{proof}
\step{1}{$\mu(S)<\infty$.}
First, consider the case when $(S,\mu)$ is a finite measure space.
  Let $(W_n)$ be a Cauchy sequence for the cut norm, with $\set{W_n}$ \ui.
By the comments before the theorem, the set \set{W_n} is relatively weakly
compact in $L^1(\qx S)$, and thus sequentially weakly compact;
hence, there exists a subsequence $W'_n=W_{k_n}$
and some $V\in L^1(\qx S)$ such that
$W'_n\to V$ weakly in $L^1(\qx S)$ as \ntoo. 
Clearly, $V$ is symmetric and thus a graphon.

In particular, whenever $T,U\subseteq S$,
\begin{equation}
  \int_{T\times U} W'_n \dd\mu^2 \to  \int_{T\times U} V \dd\mu^2.
\end{equation}
Hence, for every $m$,
recalling the definition \eqref{cn},
\begin{equation}
\Bigabs{ \int_{T\times U} (W'_m-V) \dd\mu^2} 
= \lim_\ntoo \Bigabs{ \int_{T\times U} (W'_m-W'_n) \dd\mu^2} 
\le \limsup_\ntoo \cn{W'_m-W'_n}.
\end{equation}
Taking the supremum over all measurable subsets $T$ and $U$, we obtain
\begin{equation}\label{jbx}
\cn{W'_m-V}
\le \limsup_\ntoo \cn{W'_m-W'_n}.
\end{equation}
Since $W'_n$ is a Cauchy sequence, the \rhs{} of \eqref{jbx} tends to 0 as
\mtoo, and thus  $\cn{W'_m-V}\to0$.

We have shown that the original sequence has a subsequence that converges to
$V$ for the cut norm.
Since the sequence is Cauchy, the full sequence $(W_n)$
converges to the same limit,
\ie,
$\cn{W_n-V}\to0$ as \ntoo.

Moreover, note that since $W_n'\to V$ weakly in $L^1$,
\begin{equation}\label{uriel}
  \norml{V}\le\sup_n\norml{W'_n}\le\sup_n\norml{W_n}.
\end{equation}

\step{2}{$\mu(S)=\infty$.}
In general,
  let $S=\bigcup_N S_N$, where $S_N$ is an increasing sequence of subsets of
  $S$ with finite measure.
Let again $(W_n)$ be a \ui{} Cauchy sequence for the cut norm.
Consider the restrictions
$W_{n}\NN:=W_n\restr{\qx{S_N}}$.
Then  
\begin{equation}
\cnx{S_N}{W_{m}\NN-W_{n}\NN}=\cnx{S_N}{W_m-W_n}\le\cn{W_m-W_n}, 
\end{equation}
and thus, for each $N$, 
$(W_{n}\NN)_n$ is  a Cauchy sequence for the cut norm on $S_N\times S_N$. 
Furthermore, these restrictions are \ui, since the graphons $W_n$ are.

Hence, the first part applies to $S_N$, and shows that 
for each $N$
there exists some graphon $V\NN $ on $S_N$ such that
\begin{equation}\label{ariel}
  \cnx{S_N}{W_n-V\NN}
=  \cnx{S_N}{W_n\NN-V\NN}
\to0 \qquad\text{as \ntoo}.
\end{equation}

If $M<N$, then \eqref{ariel} implies 
$ \cnx{S_M}{W_n-V\NN}\to0$ as \ntoo, and thus 
$ \cnx{S_M}{V\MM-V\NN}=0$, so
$V\MM=V\NN\restr{\qx{S_M}}$ a.e.
Consequently, there exists a symmetric measurable function $V$ on
$S\times S$ such that $V\NN=V\restrq{S_N}$ \aex{} for every $N$.
Moreover, \eqref{uriel} implies that
\begin{equation}
  \int_{\qxq{S_N}}|V|
=   \int_{\qxq{S_N}}|V\NN|
\le\sup_n \norml{W_n\NN}
\le\sup_n \norml{W_n},
\end{equation}
which is finite by \ref{wc1}.
Consequently, by monotone convergence,
\begin{equation}
  \int_{\qxq{S}}|V|
\le\sup_n \norml{W_n}<\infty,
\end{equation}
and thus $V$ is integrable and thus
a graphon.

It remains to show that $\cn{W_n-V}\to0$.
Let $T,U\subseteq S$, and let $T_N:=T\cap S_N$, $U_N:=U\cap S_N$.
It follows from \eqref{ariel} that for any fixed $N$,
\begin{equation}
  \begin{split}
 \Bigabs{ \int_{T_N\times U_N}(W_m-V)}
&=
 \Bigabs{ \int_{T_N\times U_N}(W_m-V\NN)}
=\lim_\ntoo  \Bigabs{ \int_{T_N\times U_N}(W_m-W_n)}
\\&
\le\limsup_\ntoo\cn{W_m-W_n}.
  \end{split}
\raisetag\baselineskip
\end{equation}
Letting \Ntoo, we see, by dominated convergence, that
\begin{equation}
 \Bigabs{ \int_{T\times U}(W_m-V)}
\le\limsup_\ntoo\cn{W_m-W_n},
\end{equation}
and taking the supremum over all $T$ and $U$ we obtain
\begin{equation}
\cn{W_m-V}
\le\limsup_\ntoo\cn{W_m-W_n},
\end{equation}
which tends to 0 as \mtoo.
\end{proof}
 
\begin{remark}\label{RTA}
In particular,
a relatively weakly compact set in $L^1(S^2,\mu^2)$ is
relatively   complete for the cut norm.
Moreover, 
a weakly compact set  in $L^1(S^2,\mu^2)$ is
  complete for the cut norm $\cn\cdot$,
since
the argument in the first part of the proof shows that a Cauchy sequence for
the cut norm in a weakly compact set converges to an element of that set.
Note, however, that
the identity map is, in general,  \emph{not} continuous 
$\bigpar{L^1(S^2),\text{weak}}\to \bigpar{L^1(S^2),\cn\cdot}$, 
and  that a
weakly compact set does not have to be compact for the cut norm; see
\refE{Eweakbad}. 
\end{remark}

\refT{TXA} is related to
\citet[Theorem C.7]{BCCZ14a}, which shows that a uniformly integrable set of
graphons  on $\oi$ is relatively compact (and thus relatively complete) for
the cut metric, 
see \refC{CC} below.

We assumed in \refT{TXA} uniform integrability, \ie, \ref{wc1} and
\ref{wc2}. None of these conditions suffices alone; \refE{EA1} shows that
\ref{wc1} is not enough and \refE{EA3} shows that \ref{wc2} is not enough
for relative completeness in the cut norm.
However, if we consider only non-negative graphons, then \ref{wc2} suffices.

\begin{theorem}
  \label{TWC2}
Let $(S,\mu)$ be a \gsf{} measure space.
A \sui{} set of non-negative  graphons in $L^1(S^2,\mu^2)$ 
is
relatively   complete for the cut norm.
\end{theorem}

\begin{proof}
  If $(W_n)$ is a Cauchy sequence for the cut norm on $S$, then
$\int_{\qx S}W_n$ converges;
hence,
if every $W_n\ge0$, then
\begin{equation}\label{hermia}
C_1:=\sup_n\norml{W_n}=\sup_n\int_{\qx S} W_n<\infty,
\end{equation}
so \set{W_n} satisfies \ref{wc1}.
Consequently, if \set{W_n} also is \sui, 
then it is \ui{} and \refT{TXA} implies that the sequence converges.
\end{proof}

In particular, these results yield results for $L^p$-bounded sets of graphons.
(This too is related to results in \cite[in particular Theorem  2.13]{BCCZ14a}
for the cut metric in the
case $S=\oi$.)

On a finite measure space, for example a probability space,
it is well-known, and easy to see by \Holder's inequality,
that a set that is bounded in $L^p$ for some $p>1$ is uniformly integrable,
and thus \refT{TXA} applies; thus the set is relatively complete for the cut
norm. 
In particular, a uniformly bounded set of graphons on a finite measure space
is relatively complete
for the cut norm. 
This fails for infinite measure spaces, 
see Examples \ref{EA3} and \refE{EA3p}. 
Nevertheless, an $L^p$-bounded set, for $1<p\le\infty$,
is \sui{}, which leads to the following results.
(We shall see in \refE{EA1} that the results do not hold for $p=1$.)

\begin{theorem}\label{TCA}
Let\/ $(S,\mu)$ be a \gsf{} measure space.
Let\/ $1<p\le\infty$ and $C<\infty$ 
and let $\cWpc$ be 
the set of graphons $W$ on $(S,\mu)$ with $\normlp{W}\le C$.
\begin{romenumerate}
\item \label{TCA+}
The set of non-negative graphons in $\cWpc$ is complete for the cut norm.
Hence, a set of non-negative graphons on $S$ that is bounded in $L^p$ is
relatively complete for the cut norm.

In particular, the set of \oi-valued graphons on $S$ is complete for the cut
norm. 
\item \label{TCA1}
If\/ $C_1<\infty$, then
the set of graphons $W$ in $\cWpc$ such that also $\norml{W}\le C_1$
is complete for the cut norm.
Hence, a set of graphons on $S$ that is bounded in both $L^1$ and $L^p$ is
relatively complete for the cut norm.
\item \label{TCA0}
If $\mu$ is a finite measure, then $\cWpc$ is complete for the cut norm.
Hence, a set of graphons on a finite measure space that is bounded in $L^p$ is
relatively complete  for the cut norm.
\end{romenumerate}
\end{theorem}

Before giving the proof, we give a simple lemma. It is certainly known, 
but we have not found an explicit reference so for completeness, we give a
proof. 
(Part \ref{LLp} follows in the case $S=\oi$ from the more advanced
\cite[Theorem 2.13]{BCCZ14a}, and the proof uses similar ideas as there.)

\begin{lemma}
  \label{LL}
Let $(S,\mu)$ be a \gsf{} measure space and let $W_n$ and $W$ be graphons on $S$
such that $\cn{W_n-W}\to0$ as \ntoo.
\begin{romenumerate}
\item \label{LL+}
If each $W_n\ge0$, then $W\ge0$ a.e.
\item \label{LLp}
If\/ $1\le p\le \infty$ and
each $W_n\in L^p(S,\mu)$, then $\normlp{W}\le\sup_n\normlp{W_n}$.
\end{romenumerate}
In other words, the set of non-negative graphons on $S$
and, for any $p\ge1$ and $C$, the set of graphons $W$ on $S$
with $\normlp{W}\le C$ are closed for the cut norm.
\end{lemma}
\begin{proof}
  First, consider the case when $(S,\mu)$ is a probability space.
There exists a sequence $(A_i)_{i=1}^\infty$ of measurable subsets of $S$
such  that if 
$\cF_\infty$ is the $\gs$-field generated by \set{A_i}, then $W$ is
$\qxp{\cF}$-measurable; see \eg{} the proof of \cite[Lemma 7.3]{SJ249}.

Let $\cF_N$ be the sub-$\gs$-field generated by \set{A_1,\dots,A_N}.
Then $\qx{\cF_N}$ is an increasing sequence of $\gs$-fields on $\qx S$
and their union generates $\qx{\cF_\infty}$, so the martingale limit theorem
yields
\begin{equation}
  \label{orlando}
\E\bigpar{W\mid\qx{\cF_N}}\asto \E\bigpar{W\mid\qx{\cF_\infty}}=W
\qquad\text{as \Ntoo}.
\end{equation}
Furthermore, each $\cF_N$ is finite, and generated by some partition
$\cP_N:=\set{B_{Nj}:1\le j\le m_N}$ of $S$.
The conditional expectation
$\E\bigpar{W\mid\qx{\cF_N}}$ is constant on each ``rectangle'' $B_{Ni}\times
  B_{Nj}$, and equals there, provided the rectangle has positive measure,
the average $(\mu(B_{Ni})\mu(B_{Nj}))\qw\int_{B_{Ni}\times B_{Nj}} W\dd\mu^2$.
The same holds for each $W_n$, and consequently, the assumption
$\cn{W_n-W}\to0$ implies that
\begin{equation}
  \label{celia}
\E\bigpar{W_n\mid\qx{\cF_N}}\asto \E\bigpar{W\mid\qx{\cF_N}}
\qquad\text{as \ntoo},
\end{equation}
for each fixed $N$.

For \ref{LL+}, we note that if $W_n\ge0$, then 
$\E\bigpar{W_n\mid\qx{\cF_N}}\ge0$ and thus
\eqref{celia} and
\eqref{orlando} yield $W\ge0$ a.s.
Similarly, for \ref{LLp}, if $C:=\sup_n\normlp{W_n}$, then
$\normlp{\E\bigpar{W_n\mid\qx{\cF_N}}}\le C$ and \eqref{celia} and
\eqref{orlando} yield $\normlp{W}\le C$, 
using Fatou's lemma twice if $p<\infty$ (and 
directly if $p=\infty$).

This completes the proof if $\mu(S)=1$. If $\mu(S)<\infty$, we replace $\mu$
by $\mu/\mu(S)$ and the result follows from the case just treated.

In general, 
$S=\bigcup_m S_m$, where $S_m$ is an increasing
sequence of subsets with finite measure. Since $\cn{W_n-W}\to0$ implies
$\cnx{S_m}{W_n-W}\to0$, it follows from the finite measure case that for
every $m$,
in case \ref{LL+}, $W\ge0$ \aex{} on $\qx{S_m}$ and in case \ref{LLp},
$\norm{W}_{L^p(S_m)}\le\sup_n\norm {W_n}_{L^p(S_m)}\le\sup_n\normlp{W_n}$.
The conclusions follow, for \ref{LLp} using monotone convergence 
when $p<\infty$.
\end{proof}

\begin{proof}[Proof of \refT{TCA}]
Note first that, as said above, $\cWpc$ is \sui. 
In fact, if $1<p<\infty$ and $W\in\cWpc$, then we have
$\int_{|W|>B}|W|\dd\mu\le B^{1-p}C^p$ and \ref{wc2} follows;
if $p=\infty$, take $B=C$ in \ref{wc2}.

  \pfitemref{TCA+}
The set of non-negative graphons in $\cWpc$ is relatively complete by
\refT{TWC2} and closed by \refL{LL}. Hence the set is complete.

For the final sentence, note that the set of \oi-valued graphons equals the
set of non-negative 
graphons in $\cW(\infty,1)$.

\pfitemref{TCA1}
The set $\cWpc\cap\cW(1,C_1)$ of graphons $W$ on $S$ such that 
$\normlp{W}\le C$ and $\norml{W}\le C_1$ is \ui, since \ref{wc1} is assumed
and \ref{wc2} 
follows from the $L^p$-bound as seen above. Hence, 
$\cWpc\cap\cW(1,C_1)$ is relatively complete for the cut norm by \refT{TXA}.
Furthermore, $\cWpc$ and $\cW(1,C_1)$ are closed by \refL{LL}\ref{LLp}.
Hence, $\cWpc\cap\cW(1,C_1)$ is complete.

  \pfitemref{TCA0}
If $\mu$ is a finite measure, then $\cWpc$ is $L^1$-bounded by \Holder's
inequality, and thus,
as said before the theorem, $\cWpc$ is uniformly integrable.
Hence the result follows by \refT{TXA} and \refL{LL}, or from \ref{TCA1}.
(Alternatively, it follows that $\cWpc$ is weakly compact in $L^1$;
hence the result follows by \refR{RTA}.)
\end{proof}

\subsection{Completeness for the cut metric}
It is time to turn to our main interest, the cut metric $\dcut$ defined in
\refSS{SSdcut}. 
We repeat that part \ref{TBoi} of the following theorem
was originally proved in \cite[Corollary 2.13]{BCCH16}.
Furthermore,
for the special case of graphons on \oi,
\ref{TBui} follows from \citet[Theorem C.7]{BCCZ14a}, and 
\ref{TBp+} and \ref{TBp1} follow from
\cite[Theorem 2.13]{BCCZ14a}.
(These theorems in
\cite{BCCZ14a} show much stronger results on compactness,
see Corollaries \ref{CC}--\ref{CCp} below.)

\begin{theorem}\label{TB}
\begin{thmenumerate}
\item \label{TBui}
Any \ui{} set of  graphons is relatively complete for the cut
metric. 
\item \label{TB+}
A \sui{} set of non-negative graphons 
is 
relatively complete for the cut metric.
\item \label{TBp+}
If\/ $1<p\le\infty$ and $C<\infty$, then the set of
non-negative graphons $W$  with $\normlp{W}\le C$ is complete
for the cut metric. 
Hence, a set of non-negative graphons  that is bounded in $L^p$ is
relatively complete for the cut metric.
\item \label{TBoi}
The set of\/ \oi-valued graphons is complete for
the cut metric.
\item \label{TBp1}
If\/ $1<p\le\infty$ and $C,C_1<\infty$, then the set of
 graphons $W$  with $\normlp{W}\le C$ 
and $\norml{W}\le C_1$ is complete
for the cut metric. 
Hence, a set of graphons that is bounded in both $L^1$ and $L^p$ is
relatively complete for the cut metric.
\end{thmenumerate}
\end{theorem}

We first prove a general lemma that will enable us to reduce various parts
to the corresponding claims for the cut norm.
\begin{lemma}
  \label{Lcauchy}
Suppose that $(W_n)$ is a sequence of graphons that is a Cauchy sequence
for the cut metric.
Then there exists a sequence $(W_n')$ of graphons on $\bbRp$ with $W_n\cong
W_n'$ such that $(W_n')$ is a Cauchy sequence for the cut norm
$\cnx{\bbRp}\cdot$. 
\end{lemma}
\begin{proof}
By \refP{P2.8}, we may replace the 
graphons by equivalent graphons on $(\bbRp,\leb)$, and we may thus assume that
every $W_n$ is defined on $\bbRp$.

First, suppose that
$\dcut(W_n,W_{n+1})<2^{-n}$ for all $n$.
By \refP{P4.3c}, this implies the existence of
 \mpb{s} $\gf_n:\bbRp\to\bbRp$ such that
 \begin{equation}
   \label{lear}
\cn{W_n-W_{n+1}\qphi{n}}<2^{-n}.
 \end{equation}

Let $\psi_n:=\gf_{n-1}\circ\dotsm\circ\gf_1$ (with $\psi_1$ the identity); 
this is a \mpb{} $\bbRp\to\bbRp$. 
Then $\psi_{n+1}=\gf_n\circ\psi_n$, and thus
$W_{n+1}\qpsi{n+1}=W_{n+1}^{\gf_n\circ\psi_n}=(W_{n+1}\qphi{n})\qpsi{n}$.
Hence, \eqref{lear} implies
\begin{equation}\label{kent}
  \begin{split}
	  \cn{W_n\qpsi{n}-W_{n+1}\qpsi{n+1}}
&=  \cn{W_n\qpsi{n}-(W_{n+1}\qphi{n})\qpsi{n}}
=  \cn{(W_n-W_{n+1}\qphi{n})\qpsi{n}}
\\&
=  \cn{W_n-W_{n+1}\qphi{n}}<2^{-n}.
  \end{split}
\end{equation}
Consequently, the sequence $(W_n\qpsi n)$ is a Cauchy sequence for the cut
norm $\cn\cdot$ on $\bbRp$, so we may take $W_n':=W_n\qpsi n$.

In general,
we can select a subsequence $n_k$ such that 
\begin{equation}
  \label{brutus}
\dcut\bigpar{W_{n_k},W_m}<2^{-k}
\qquad \text{for all $k\ge1$ and $m\ge n_k$}.
\end{equation}
In particular, $\dcut(W_{n_k},W_{n_{k+1}})<2^{-k}$, so the case just treated
applies to the subsequence $(W_{n_k})$ and shows the existence of graphons
$W_{n_k}'\cong W_{n_k}$ defined on $\bbRp$ such that
$\cn{W_{n_k}'-W_{n_{k+1}}'}<2^{-k}$.
Moreover, for any $m\in(n_k,n_{k+1})$, \eqref{brutus} implies
$\dcut(W_{n_k}',W_m)=\dcut(W_{n_k},W_m)<2^{-k}$,
and thus \refP{P4.3c} shows that there exists a \mpb{}
$\gf_m$ such that $\cn{W_{n_k}'-W_m\qphi m}<2^{-k}$; let $W_m':=W_m\qphi m$.
It follows that $(W_n')$ is a Cauchy sequence for the cut norm.
(Actually, for our purposes it would be enough to consider the subsequence
$W'_{n_k}$; the full statement and the last part of the proof is only for
completeness.)
\end{proof}

We also need a version of \refL{LL} for the cut metric.
\begin{lemma}
  \label{LLgd}
 Let $W_n$ and $W$ be graphons
such that $\dcut\xpar{W_n,W}\to0$ as \ntoo.
\begin{romenumerate}
\item \label{LLgd+}
If each $W_n\ge0$, then $W\ge0$ a.e.
\item \label{LLgdp}
If each $W_n\in L^p(S,\mu)$, then $\normlp{W}\le\sup_n\normlp{W_n}$.
\end{romenumerate}
In other words, the set of non-negative graphons
and, for any $C$, the set of graphons $W$
with $\normlp{W}\le C$ are closed for the cut metric.
\end{lemma}

\begin{proof}
As above, 
by \refP{P2.8}, we may replace the 
graphons by equivalent graphons on $(\bbRp,\leb)$, and  assume that
$W$ and every $W_n$ is defined on $\bbRp$,
using also   \refE{E=}\ref{E=>0}\ref{E=Lp}.
By \refP{P4.3c}, this implies the existence of
pull-backs $W_n':=W_n\qphi n\cong W_n$ such that 
$\cn{W-W_n'}< \dcut(W,W_n)+1/n\to0$ as \ntoo.
Now apply \refL{LL}.
\end{proof}

\begin{remark}\label{RLgd}
  By considering suitable subsequences, it follows that the conclusion in 
\refL{LLgd}\ref{LLgdp} can be improved to
$\normlp{W}\le\liminf_n\normlp{W_n}$.
\end{remark}

\begin{proof}[Proof of \refT{TB}]
\pfitemref{TBui}
Suppose that $(W_n)$ is a  sequence of graphons (possibly defined on different
measure spaces) that is \ui{} and a Cauchy
sequence for the cut metric.

By \refL{Lcauchy}, there exist 
$W_n'\cong W_n$ such that $W_n'$ are defined on $\bbRp$ and
the sequence $(W_n')$ is a Cauchy sequence for the cut
norm $\cn\cdot$ on $\bbRp$.  Note that replacing $W_n$ by
the equivalent 
$W_n'$ preserves uniform integrability;
this follows from \refT{TC=} applied to
\ref{wc1} and \ref{wc2} with explicit bounds.

Hence, \refT{TXA} shows that there exists a graphon $W$ on $\bbRp$
such that 
$W_n'$ converges to $W$ in cut norm as \ntoo.
This implies
\begin{equation}
  \dcut(W_n,W)\le \cn{W_n'-W}\to0 \qquad\text{as \ntoo},
\end{equation}
and thus $W_n\to W$ in the cut metric.

\pfitemref{TB+}
Argue as in the proof of \ref{TBui}, now using \refT{TWC2}.

\pfitemref{TBp+}
Argue as in the proof of \ref{TBui}, now using
\refT{TCA}\ref{TCA+} and 
\refL{LLgd}.
(Alternatively, note that \eqref{hermia} holds in this case too, and use
\ref{TBp1}.)

\pfitemref{TBoi}
This is just a special case of \ref{TBp+}.

\pfitemref{TBp1}
The set is \ui, as seen in the proof of \refT{TCA},
so it is relatively complete for the cut metric by \ref{TBui}.
Moreover, the set is closed by \refL{LLgd}.
(Alternatively, one could use \refT{TCA}\ref{TCA1}.)
\end{proof}

\refT{TB} is our main  result about completeness. Note, however,
that the conditions there are not necessary. In particular, as shown in
\refE{Enotui}, (semi)uniform integrability is not necessary for convergence and
completeness in cut metric.

On the other hand, some conditions are needed, and we give a number of
(counter) examples to illustrate that.
In particular,  \refE{EA1} shows that the set of all non-negative
graphons is not complete, and also that \refT{TB}\ref{TBp+} does not hold for
$p=1$; furthermore,
\refE{EA3} shows that \refT{TB}\ref{TB+} and  \ref{TBp+} do not hold without the
assumption that
the graphons are non-negative, even if we assume uniform boundedness (\ie,
 $p=\infty$). 

\begin{remark}
For a given measure space $(S,\mu)$, the cut norm is equivalent to the
injective tensor product norm in $L^1\itensor L^1$, which is given by taking
the supremum over all $g,h:S\to[-1,1]$ in \eqref{macduff}, see \eg{}
\cite[Chapter 3]{Ryan} or
\cite[Remark 4.2]{SJ249}.
Hence every Cauchy sequence for the cut norm converges in the completed
injective tensor product $L^1(S)\itensor L^1(S)$. 
However, the limit may lie outside $L^1(S\times S)$.
In fact, $L^1(S\times S)$ is a dense subspace of
$L^1(S)\itensor L^1(S)$, but typically (\eg{} for
$S=\oi$ or $\bbRp$) the norms are not 
equivalent, as is witnessed \eg{} by $V_n$ in \refE{EA3} below, 
and thus $L^1(S\times S)\subsetneq L^1(S)\itensor L^1(S)$.
Hence there exists Cauchy sequences for the cut norm (and thus also for the cut
metric) with limits
not in $L^1(S\times S)$.
It is also easy to see that such sequences can be made symmetric, \ie, graphons.

On an abstract level, the failure of completeness in general for the cut
norm is thus almost
obvious. The examples below give some simple concrete examples.
(In these examples, the Cauchy sequences thus have limits in 
$ L^1(S)\itensor L^1(S)\setminus L^1(S\times S)$.)

See further \eg{} \cite{Ryan} and note,
in contrast, 
that the completed projective tensor product
$L^1(S)\ptensor L^1(S)$ equals  $L^1(S\times S)$.

One might be tempted to extend the definition of graphons to include
all symmetric elements of $L^1(S)\itensor L^1(S)$. However, we doubt that
this is useful. In particular, we do not see any way to define random graphs
generalizing the construction in \refSS{SSRG} unless $W$ is a function.
\end{remark}

We say that a  graphon $W$ on $S=\oi$ or $\bbRp$ is a \emph{step} graphon if
there is a 
partition of $S$ into a finite number of intervals $I_i$ (the steps) 
such that $W$ is
constant on each $I_i\times I_j$.
(Note that the definition in \cite{SJ249} is more general.)

\begin{example}[An $L^1$-bounded  set of  non-negative graphons on $\oi$
that is not relatively complete]\label{EA1}
This example is essentially the same as 
\citet[Proposition 2.12b]{BCCZ14a} but stated differently; 
for completeness we give full details.

Note that for each $n\ge1$ there exists a step graphon $U_n$ on $\oi$
with steps of equal measure,
values in \set{0,1}, 
and
\begin{equation}\label{cordelia}
\int U_n=\tfrac12,
\qquad
\cn{U_n-\tfrac12}<4^{-n}.
\end{equation}

One way to see this is to consider the \ER{}
random graph $G(N,M)$ with $M=N^2/4$ for large even $N$. 
Then, as \Ntoo, the graphon $W_{G(N,M)}$ converges in probability to the
constant graphon $\frac12$. Consequently, we may take $U_n=W_G$ where $G$ is
a realization of $G(N,N^2/4)$ for some large $N=N(n)$. (If $N$ is chosen
large enough, most realizations will do.)
Alternatively we may take $W_G$ where $G$ is a realization of
$G(N,\frac12)$, or a sufficiently large deterministic quasi-random graph
such as a Paley graph, adjusted (arbitrarily) to have exactly $N^2/4$ edges.

We construct inductively a sequence of step graphons $W_n$ on \oi,
with values in
$\set{0,2^n}$.
Let $W_0=1$. Given $W_{n}$, denote its steps by $I_{n,i}$.
On each rectangle (actually square) $I_{n,i}\times I_{n,j}$ where
$W_n$ is non-zero, and thus equal to $2^{n}$, 
let $W_{n+1}$ be  a scaled copy of $2^{n+1} U_{n+1}$; on the other rectangles,
where $W_n=0$, let $W_{n+1}=0$.
In other words: Let $U_n$ have $m_n$ steps of length $m_n\qw$, and let
$M_n:=\prod_{i=1}^n m_i$. 
Then, let 
$\hU_n(x,y):= U_n(\frax{M_{n-1} x},\frax{M_{n-1} y})$
and $W_n:=2^n\prod_{j=1}^n \hU_j$.

It is easily seen that 
$\cn{W_n-W_{n+1}}<2^{-n}$, so $(W_n)$ is a
Cauchy sequence for the cut norm, and thus for the cut metric.
However, $W_n$ does not converge in the cut metric (and thus also not in the cut
norm).
To see this, suppose that $\dcut(W_n,W)\to0$ for some graphon $W$.
Then $W\ge0$ a.e.
We do not assume that $W$ is defined on \oi, but we may and shall
assume that is defined on $\bbRp$, see \refP{P2.8}.
We also extend each $W_n$ trivially to $\bbRp$, adding another step
$I_{n,0}=(1,\infty)$. 
Then, by \refP{P4.3c}, there exist \mpb{s}
$\gf_n:\bbRp\to\bbRp$ such that $\cn{W_n-W^{\gf_n}}<\dcut(W_n,W)+1/n\to 0$.

For any $N\ge1$,
on each rectangle $I_{N,i}\times I_{N,j}$ where
$W_N=0$, we have $W_n=0$ for all $n\ge N$. Call such rectangles \emph{good}
(for $N$).
Then, on each good rectangle $Q=I_{N,i}\times I_{N,j}$, for $n\ge N$,
\begin{equation}\label{prospero}
{\int_Q W^{\gf_n}}={\int_Q (W^{\gf_n}-W_n)}
\le\cn{W\qphi n-W_n} \to 0.   
\end{equation}
Let $A_N$ be the union of all good rectangles for a
given $N$, and let $B_N:=\bbRp^2\setminus A_N$; note that $\leb^2(B_N)=2^{-N}$
since $1=\int W_N=2^N\lebb(B_N)$.
Given any $\eps>0$, we can find $N$ such that if $B$ is a set with
$\leb^2(B)<2^{-N}$, then 
$\int_B W <\eps$. This implies that
$\int_{B_N} W^{\gf_n} = \int_{\gf_n\tensor\gf_n(B_N)}W<\eps$ for every $n$.
Furthermore, \eqref{prospero} implies that $\int_{A_N} W^{\gf_n}\to0$ as \ntoo.
Since
\begin{equation}
  \int_{\bbRp^2} W = \int_{\bbRp^2} W\qphi{n}
=\int_{A_N} W\qphi{n}+\int_{B_N} W\qphi{n},
\end{equation}
it follows by letting \ntoo{} that  $\int_{\bbRp^2} W\le\eps$. Since $\eps$ is
arbitrary, this implies $\int W=0$ (and thus $W=0$ \aex).

On the other hand, we have 
$\int W_n=1$ for each $n$, and it follows from $\dcut(W_n,W)\to0$
that $\int W=1$,
a contradiction.

This shows that $(W_n)$ is a Cauchy sequence for the cut metric that does not
converge. 

Note that $\norml{W_n}=1$, so this example shows that we cannot take $p=1$
in \refT{TCA}\ref{TCA+}--\ref{TCA0} or \refT{TB}\ref{TBp+},\ref{TBp1}.
\end{example}

\begin{example}[the set of  {$[-1,1]$}-valued graphons on $\bbRp$
	is not complete] \label{EA3}
  Let $S=\bbRp$. For $n\ge1$, let $V_n$ be a graphon on \oi{} with values in
  $\set{\pm1}$, and thus $\norml{V_n}=1$, such that
$\cn{V_n}<2^{-n}$. For example, we can take $V_n:=2U_n-1$ with $U_n$ as in
  \refE{EA1}. 

Let $\tV_n(x,y):=V(x-n+1,y-n+1)$ when $(x,y)\in(n-1,n]^2$ and 0
  otherwise, \ie, $\tV_n$ is $V_n$ translated to $(n-1,n]^2$.
Finally, let $W_n:=\sum_{k=1}^n \tV_k$.
Then 
\begin{equation}\label{lysander}
\cn{W_n-W_{n-1}}=\cn{\tV_n}=\cn{V_n}<2^{-n}, 
\end{equation}
and thus $(W_n)$ is a  Cauchy sequence for the cut norm.

However, there is no graphon $W$ such that $\cn{W_n-W}\to0$.
In fact, suppose that this holds for some $W$. 
Then, for any integer $N$, 
\begin{equation}
  \cnx{[0,N]}{W_n-W}
=
  \cn{(W_n-W)\etta_{\qxq{[0,N]}}}
\le   \cn{W_n-W}\to0
\end{equation}
as \ntoo. On the other hand, for $N\ge n$, 
$W_n=W_N$ on $\qxq{[0,N]}$.
Hence,
$\cnx{[0,N]}{W_N-W}=0$, and thus $W=W_N$ \aex{} on $[0,N]^2$.
Consequently,
\begin{equation}\label{helena}
  \int_{[0,N]^2}|W|
= \int_{[0,N]^2}|W_N|
=\sum_{k=1}^N \int_{[0,N]^2}|\tV_k|=N.
\end{equation}
Letting \Ntoo, we see that $\int|W|=\infty$, which contradicts our
assumption that $W$ is a graphon and thus integrable.
Consequently, the uniformly bounded Cauchy sequence $(W_n)$ does not converge
in the cut norm.

We have so far considered the cut norm; we now show that the same properties
hold for the cut metric. 
It follows from \eqref{lysander} that $(W_n)$ also is a Cauchy sequence for
the cut metric. 

Suppose that $\dcut(W_n,W)\to0$ for some graphon $W$.
We may suppose that $W$ too is defined on $\bbRp$. Then, by
\refP{P4.3c}, there exist \mpb{s} $\gf_n:\bbRp\to\bbRp$
such that 
\begin{equation}
\cn{W_n-W\qphi n}<\dcut(W_n,W)+1/n\to0  
\end{equation}
as \ntoo.
Taking the restrictions to $\qx{[0,N]}$, we see that for any $n\ge N$, 
\begin{equation}
\cn{W_N-W\qphi n\restr{\qx{[0,N]}}}
=
\cn{(W_n-W\qphi n)\restr{\qx{[0,N]}}}
\le \cn{W_n-W\qphi n}\to0  
\end{equation}
and thus, by \refL{LL}\ref{LLp}, for every $N\ge1$,
\begin{equation}
  \norml{W_N}\le \sup_n \norml{W\qphi n\restr{\qx{[0,N]}}}
\le \sup_n \norml{W\qphi n}=\norml{W}.
\end{equation}
However, $\norml{W_N}=N$ by \eqref{helena}, a contradiction.

Consequently, $(W_n)$ is also for the cut metric a Cauchy sequence that does
not converge.
\end{example}

\begin{example}[An $L^p$-bounded set that is  not relatively complete]
  \label{EA3p}
Let $V_n$ be as in \refE{EA3} and let
$\VV_n:=2^{-{2n}}\striip{V_n}{2^{2n}}$,
see \eqref{str2}.
Define now $\tV_n(x,y):=\VV_n(x-2^n,y-2^n)$ on $(2^n,2^{n+1}]^2$ and 0
  elsewhere, and let again $W_n:=\sum_{k=1}^n\tV_k$.
We have, as in \eqref{lysander},
\begin{equation}
  \cn{\tV_n}=\cn{\VV_n}=2^{-2n}\cn{\striip{V_n}{2^{2n}}}=\cn{V_n}<2^{-n}.
\end{equation}
Furthermore, for any $p\ge1$,
\begin{equation}\label{demetrius}
  \int|\tV_n|^p
=  \int|\VV_n|^p
=2^{-2pn}\leb([0,2^n])^2
=2^{2(1-p)n};
\end{equation}
in particular,
$  \norml{\tV_n}=1$,
and the same argument as in \refE{EA3} (now using $[0,2^{N+1}]^2$) shows that
$(W_n)$ is a Cauchy sequence in cut norm and cut metric that does not converge
to any graphon.

Moreover, for any $p>1$, by \eqref{demetrius},
\begin{equation}
  \int|W_n|^p=\sum_{k=1}^n\int|\tV_k|^p
\le \sumk 2^{-2(p-1)n}<\infty,
\end{equation}
and thus the sequence $(W_n)$ is $L^p$-bounded for every $p>1$.
\end{example}

\begin{example}[weak convergence does not imply convergence in cut norm or
	cut metric]  \label{Eweakbad}
Let $h_n:=\sgn(\sin(2^n\pi x))$ on $S=\oi$ (the Rademacher functions).
Define
\begin{equation}
  W_n(x,y):=
  \begin{cases}
	h_n(x),& 0<x<\frac12<y<1, \\
	h_n(y),& 0<y<\frac12<x<1, \\
0,&\text{otherwise}.
  \end{cases}
\end{equation}
Then $h_n\wto 0$ weakly in $L^1(\oi)$ as \ntoo, and it follows easily that
$W_n\wto 0$ weakly in $L^1(S^2)$.
However, if $T_n:=\set{x\in(0,\frac12):h_n(x)>0}$, then
\begin{equation}
  \cn{W_n-0}=\cn{W_n}\ge\int_{T_n\times(\frac12,1)} W_n =\frac18.
\end{equation}
Hence $W_n\not\to0$ in the cut norm.

Moreover, it is easily seen that for $n\ge2$,
$W_n$ is the pull-back $W_2\qphi n$ of $W_2$ by the
\mpm{} $\gf_n$ defined by $\gf_n(x)=\frac12\frax{2^{n-1}x}$ for $x\le \frac12$,
$\gf(x)=x$ for $x>\frac12$. Hence, $W_n\cong W_2$ and, trivially, $W_n\to
W_2\neq0$ for the metric $\dcut$. Consequently,
$W_n\not\to0$ also for the cut metric $\dcut$.

It follows that the set $\set{W_n}_{n\ge1}\cup\set0$ is weakly
compact, but discrete and thus not compact for both the cut norm and the cut
metric.

We can get a similar example with \oi-valued graphons by considering
$\tW_n:=(W_n+1)/2$.
\end{example}

\begin{example}[uniform integrability is not necessary I]\label{Enotui}
Similarly to the construction of $U_n$ in \refE{EA1}, we can for each $n$
find a step graphon $V_n$ on \oi{} 
with steps of equal measure,
values in \set{0,1}, 
and
\begin{equation}\label{regan}
\cn{V_n-\tfrac1n}<4^{-n}.
\end{equation}
For example, we  can take $V_n$ as $W_G$ for a realization of $G(N,1/n)$ for
a sufficiently 
large $N=N(n)$.

If we further define $\xV_n:=nV_n$, then by \eqref{regan}, 
$\xV_n$ is \set{0,n}-valued and
\begin{equation}\label{goneril}
\cn{\xV_n-1}<n4^{-n}<2^{-n}.
\end{equation}
Hence, $\xV_n\to1$ as \ntoo{} for the cut norm, and thus also for the cut
metric. 
However, since $\xV_n$ is \set{0,n}-valued, for any $B$ and all $n>B$,
\begin{equation}
  \int_{\xV_n>B}\xV_n = \int_{\oi^2}\xV_n \to 1,
\end{equation}
where we also used \eqref{goneril}. Hence, \ref{wc2} does not hold
for \set{\xV_n}, so the set is not \ui, and not even \sui.

Consequently, the set $\set{\xV_n}_{n\ge1}\cup\set{1}$ is compact
and complete for both the cut norm on $\oi$ and the cut metric, but not \ui.

Note that in this example, the graphons are all non-negative and defined on \oi,
but unbounded. 
\end{example}

\begin{example}[uniform integrability is not necessary II]\label{Enotui2}
We obtain a related example, where the graphons are $[-1,1]$-valued but
defined on $\bbRp$, by letting $V_n$ be as in \refE{EA3} and taking the
stretched graphons $\bW_n:=\strii {V_n}n$. 
Then, \cf{} \eqref{edmund}, $\norml{\bW_n}=n$ and $\cn{\bW_n}<n2^{-n}$, so
$\bW_n\to0$ in cut norm and thus in cut metric, but $\bW_n$ is not \ui,
since \ref{wc1} does not hold.
  
Note that in this example, in contrast to \refE{Enotui},
the graphons are uniformly bounded, but take negative values and are defined
on an infinite measure space.
\end{example}

\begin{remark}
  We cannot combine the assertions of Examples \ref{Enotui} and
  \ref{Enotui2}.
First, a uniformly bounded set of graphons defined on \oi{} is automatically
\ui.
Secondly, a sequence of non-negative graphons that converges in the cut
norm or cut metric is bounded in $L^1$; hence, if it also is uniformly
bounded, then it is \ui.
\end{remark}

\subsection{Completeness for $\gd_1$ and $\gd_p$}\label{SScomplete1p}
Our main interest is the cut metric $\dcut$, but we also give a
simple corresponding result for $\gd_1$ and $\gd_p$.

\begin{theorem}\label{TYB}
  \begin{romenumerate}
  \item 
  The set of all graphons is complete for the metric $\gd_1$.
  \item \label{TYBp}
For any $p>1$, and any $C<\infty$,  the set of non-negative graphons $W$
in $L^p$ such that $\norml W\le C$ is complete for the metric $\gd_p$.
  \end{romenumerate}
\end{theorem}

\begin{proof}
  Both parts follow by the same argument as in the proof of \refT{TB}, using
  \cite[Remark 4.4]{BCCH16} to see that a $\gd_p$-version of \refL{Lcauchy}
  holds for $p\ge1$,
and (instead of \refT{TXA})
the fact that $L^1(\bbRp^2)$ and $L^p(\bbRp^2)$
  are complete; for \ref{TYBp} also the fact that if $W_n\to W$ in $L^p$ and
   $\norml{W_n}\le C$ for all $n$, then $\norml W\le C$.
\end{proof}

The $L^1$-boundedness in \ref{TYBp} is necessary in general, as is seen by the
following example. 

\begin{example}[the set of non-negative graphons is not complete for
	$\gd_p$, $p>1$]
Let $f(x):=1/(x+1)$ and define the graphon $W$ on $\bbRp$ by
$W(x,y):=f(x)f(y)$. Let further $W_n:=W\etta_{\qxq{[0,n]}}$.
Then, for every $p>1$, $f\in L^p(\bbRp)$ and thus $W\in L^p(\bbRp^2)$ and
$W_n\to W$ in $L^p$, and thus in $\gd_p$. Hence $(W_n)$ is a Cauchy sequence
of graphons for $\gd_p$. However, $W\notin L^1$, so $W$ is not a graphon.
Furthermore, $W_n$ cannot have any other limit $V$ that is a non-negative
graphon.
To see this, suppose that $V$ is a non-negative graphon and that
$\gd_p(W_n,V)\to0$. We may, as usual, assume that $V$ is defined on $\bbRp$.
By \cite[Proposition 4.3(c) and Remark 4.4]{BCCH16}, there exist \mpb{s}
$\gf_n:\bbRp\to\bbRp$ such that $\normlp{W_n-V\qphi n}\to0$.
Consequently,
\begin{equation}
  \normlp{W-V\qphi n}\le   \normlp{W-W_n}+  \normlp{W_n-V\qphi n}\to0
\end{equation}
as \ntoo. Hence $V\qphi n\to W$ in $L^p$, but this implies
\begin{equation}
  \norml{W}\le\sup_n\norml{V\qphi n}=\norml{V}<\infty,
\end{equation}
a contradiction.
\end{example}

\begin{remark}
If we only consider non-negative $L^p$ graphons on probability
spaces as in \cite{BCCZ14a}, then the set of all of them is complete for
$\gd_p$, for any $p>1$; this 
follows since a Cauchy sequence $(W_n)$
has
\begin{equation}
  \sup_n\norml{W_n}
\le   \sup_n\normlp{W_n}
=\sup_n\gd_p(W_n,0)<\infty
\end{equation}
and thus $W_n$ converges by \refT{TYB}\ref{TYBp}.
\end{remark}

\section{Compactness}\label{Scompact}

  \citet[Theorem 2.12]{BCCH16} 
prove a partial characterisation of
relatively compact sets for the cut metric.
Their theorem is stated in terms on convergent (sub)sequences, but 
it implies immediately a statement on relative compactness, 
\viz{} the case of uniformly bounded graphons in \refT{TC} below.
(In fact, \cite[Theorem 2.12]{BCCH16}  is equivalent to this 
compactness result together with the completeness result in
\refT{TB}\ref{TBoi}.)

We give a new proof of their result and extend it in several ways; 
in particular,
we give extensions
from uniformly bounded to uniformly integrable
graphons, and  from non-negative to signed
graphons. 
(Recall, however, that uniform integrability is not needed for convergence,
and thus not for compactness, 
not even for graphons on \oi, see \refE{Enotui}.) 

We begin with some definitions.

We say, as in \cite[Definition 2.11]{BCCH16} that a set $\cW$ of graphons
has \emph{\urt} if 
for every $\eps>0$, there exists $M<\infty$ such that for every graphon
$(W,S,\mu)\in\cW$, there exists a set $U\subseteq S$ such that $\mu(U)\le M$
and 
\begin{equation}
  \label{pyramus}
\norml{W-W\etta_{\qx U}}<\eps.
\end{equation}
Similarly, we say that a set $\cW$ of graphons has \emph{\ucrt} if
for every $\eps>0$, there exists $M<\infty$ such that for every graphon
$(W,S,\mu)\in\cW$, there exists a set $U\subseteq S$ such that $\mu(U)\le M$
and 
\begin{equation}
  \label{thisbe}
\cn{W-W\etta_{\qx U}}<\eps.
\end{equation}
I.e., we relax the $L^1$-norm in \eqref{pyramus} to the cut norm.
For non-negative graphons the two notions are equivalent, as shown in the
following simple lemma.
(\refE{E-urt} shows that the notions differ in general.)

\begin{lemma}\label{Lurt}
  \begin{thmenumerate}
  \item \label{Lurt1}
 Any set of graphons with \urt{} has \ucrt.
  \item \label{Lurt2}
A set of non-negative graphons has \urt{} if and only if it has \ucrt.
  \end{thmenumerate}
 \end{lemma}

\begin{proof}
  \pfitemref{Lurt1}
By \eqref{cutl1}.

\pfitemref{Lurt2}
If $W\ge0$, then also $W-W\etta_{\qx U}\ge0$ and thus
$\cn{W-W\etta_{\qx U}}=\norml{W-W\etta_{\qx U}}$.
\end{proof}

We say that a set $\cW$ of graphons is \emph{\ucr} if
for every $\eps>0$, there exists $B<\infty$ such that for every graphon
$(W,S,\mu)\in\cW$, there exists a graphon $V$ on $ S$ such that $|V|\le B$
and 
\begin{equation}
  \label{quince}
\cn{W-V}<\eps.
\end{equation}
Note that the corresponding notion with $\norml\cdot$ in \eqref{quince} is
equivalent to \ref{wc2}.
Hence a \sui{} set is \ucr.

\begin{remark}\label{Ruur}
  \citet[Definition C.2]{BCCZ14a} give, for graphons on \oi,
a definition of what they call
\emph{uniformly upper regular} sequences of graphons and use this for 
(essentially) a
characterisation of sequential compactness in 
\cite[Theorems C.13 and  C.5]{BCCZ14a}.
Our definition of \ucr{} is quite different, but it is used for a similar
purpose. It seems interesting to investigate the relation between these
notions, but we have not yet done so, and leave it as an open problem.
\end{remark}

Recall that if $A$ is a set in a metric space $(X,d)$, then
an \emph{$\eps$-net} for $A$
is a subset $Y\subseteq X$ such that for every $x\in A$
there exists $y\in Y$ with $d(x,y)<\eps$.
The set $A$
 is \emph{totally  bounded} if for every $\eps>0$ there  exists a finite
 $\eps$-net for $A$. (We may here further assume that the $\eps$-net is a
 subset of $ A$, but we
find it more convenient not to do so.)
Recall also that 
$A$ is compact if and only if it is complete and totally bounded; 
it follows that
 $A$ is relatively compact if and only if it is relatively complete and
totally bounded.

We can now characterise totally bounded sets of graphons.
Note that condition in \ref{TD1} below is a combination of the two conditions
in \ref{TD2}.
\begin{theorem}\label{TD}
  Let $\cW$ be a set of graphons. Then the following are equivalent.
  \begin{romenumerate}
  \item \label{TDtb}
$\cW$ is totally bounded for the cut metric.
  \item \label{TD2}
$\cW$ is \ucr{} and has \ucrt.
  \item \label{TD1}
For every $\eps>0$ there exists $M<\infty$ and $B<\infty$ such that for
every graphon $(W,S,\mu)\in\cW$ there exists a set $U\subseteq S$ 
with $\mu(U)\le M$ and a
graphon $V$ on $S$ such that $|V|\le B\etta_{\qx U}$ and $\cn{W-V}<\eps$. 
  \end{romenumerate}
\end{theorem}

We postpone the proof to the next section, and give first some consequences
of the theorem. We consider two different simplifying assumptions in the
following two subsections.

\subsection{Compactness, the \sui{} case}\label{SScompactsui}
In this subsection, we consider for simplicity
only sets of
graphons that are 
\sui, 
for example sets that are \ui{} or uniformly bounded.
Such sets are always \ucr, as said above, 
since we may take $V:=W\etta_{|W|\le B}$ in
\eqref{quince} for some large $B$.
Hence, \refT{TD} implies the following.

\begin{corollary}\label{CD}
  A \sui{} set $\cW$ of graphons 
is totally bounded for
  the cut metric if and only   if it has \ucrt.
\nopf
\end{corollary}


We combine \refC{CD} with results on completeness in \refS{Scomplete} to
obtain 
results on compactness. 

\begin{theorem}\label{TUC}
  A \ui{} set of graphons is relatively compact for the cut metric if and
  only if it has \ucrt.
\end{theorem}

\begin{proof}
  By \refT{TB}\ref{TBui} and \refC{CD}.
\end{proof}

\begin{theorem}\label{TUC+}
  A \sui{} set of non-negative graphons 
is relatively compact for the cut metric if and
  only if it has \ucrt.
\end{theorem}
\begin{proof}
  By \refT{TB}\ref{TB+} and \refC{CD}.
\end{proof}

\citet[Theorem 2.12]{BCCH16} give a similar result 
(for uniformly bounded graphons)
using \urt{} instead of
\ucrt{}; we obtain a 
(reformulation of) their result as a corollary, where we
furthermore extend their result from 
uniformly bounded graphons to \sui{} graphons. 

\begin{theorem}[Mainly {\cite[Theorem 2.12]{BCCH16}}]\label{TC}
  Let $\cW$ be a \sui{} set of graphons.
(For example a uniformly bounded set, or a set bounded in $L^p$ for some $p>1$.)
\begin{romenumerate}
\item \label{TC1}
If $\cW$ has \urt, then  $\cW$ 
is relatively compact for the cut metric.
\item \label{TC2}
The converse holds if all graphons in $\cW$ are non-negative.
\end{romenumerate}
\end{theorem}

We prove first a simple lemma.

\begin{lemma}
  \label{LX}
A set of graphons that is \sui{} and has \urt{} is \ui.
\end{lemma}

\begin{proof}
  Denote the set by $\cW$.
By the definition \ref{wc2} of \sui, there exists $B<\infty$ such that
$\int_{|W|>B}|W|<1$ for every $W\in\cW$.
Furthermore, take $\eps=1$ in the definition of \urt. Thus, there exists
$M<\infty$ such that if $W\in\cW$ is defined on $(S,\mu)$, then there exists
$U\subseteq S$ with $\mu(U)\le M$ and $\norml{W-W\ettaqx{U}}<1$.
Consequently,
\begin{equation}
  \begin{split}
	\norml{W}&
=\norml{W-W\ettaqx{U}}+\int_{\qx U}|W|
\\&
\le 1+\int_{|W|>B}|W|+\int_{\qx U}B
\le 1+1+M^2B.
  \end{split}
\raisetag\baselineskip
\end{equation}
Hence, \ref{wc1} holds, so $\cW$ is \ui.
\end{proof}

\begin{proof}[Proof of \refT{TC}]

\pfitemref{TC1}
By \refL{LX}, $\cW$ is \ui, and $\cW$ has \ucrt, so $\cW$ is relatively
compact by \refT{TUC}.

\pfitemref{TC2}
If $\cW$ is relatively compact, then it is totally bounded,
and \refT{TD} shows that it has \ucrt, and
the result follows by \refL{Lurt}\ref{Lurt2}.
\end{proof}

\refE{EEE} below shows that a \sui{} set of graphons with \ucrt{} does not
have to be relatively compact. Note that Theorems \ref{TUC}, \ref{TUC+} and
\ref{TC}\ref{TC1} strengthen the assumption in three different ways (\ui,
non-negative and \urt, respectively), and that we thus need these stronger
assumptions.

Furthermore, 
\refE{EF} shows that uniform integrability is not enough to imply
relative compactness, even for \oi-valued graphons; hence the first
condition in Theorems \ref{TUC}--\ref{TC} does not suffice to imply the other
conditions.

\refE{E-urt} shows that a compact set of signed graphons does
not have to have \urt, so \refT{TC}\ref{TC2} does not hold without assuming
non-negativity.

\begin{example}[A \sui{} set with \ucrt{} that is not relatively compact]
\label{EEE}
  The sequence $(W_n)$ in \refE{EA3} is a Cauchy sequence for the cut
  metric, and thus it is totally bounded; 
hence it has \ucrt{} by \refT{TD}.
Furthermore, the sequence is
  uniformly bounded and is thus \sui. Nevertheless, the sequence does not
  converge, so there can be no convergent subsequence and thus the \set{W_n}
  is not relatively compact.

Thus, we cannot replace \ui{} by \sui{} in \refT{TUC}, 
non-negative by arbitrary (signed) in \refT{TUC+}, or
\urt{} by \ucrt{} in \refT{TC}.
\end{example}

\begin{example}[A \ui{} sequence that is not relatively compact]
  \label{EF}
Let $W_n$ be the graphon on $\bbRp$ given by  $W_n:=n\qww\ettaqx{[0,n]}$.
The graphons $W_n$ are \oi-valued and thus uniformly bounded; furthermore,
$\norml{W_n}=1$ so the set $\set{W_n}$ is also $L^1$-bounded and thus \ui.

However, $\norm{W_n}_\infty\to0$ as \ntoo, and it follows by \refR{RLgd}
that if $\dcut(W_n,W)\to0$ for some subsequence, then $\norm{W}_\infty=0$,
so $W=0$ \aex; however, this is impossible since $\dcut(W_n,0)=\int W_n=1$.
Henc, no subsequence converges, and thus \set{W_n} is not relatively
compact.

By \refT{TD}, \set{W_n} cannot have \ucrt.
\end{example}

\begin{example}[convergence does not imply \urt{} without non-negativity]
  \label{E-urt}
Let $V_n$ be as in \refE{EA3} and both stretch and rescale them to
$W_n:=n\qw\strp{V_n}n$.
Then, see \eqref{edmund}, 
\begin{equation}\label{sycorax}
  \norml{W_n}=n\qw\norml{\strp{V_n}n}=\norml{V_n}=1,
\end{equation}
while
\begin{equation}\label{caliban}
  \cn{W_n}=n\qw\cn{\strp{V_n}n}=\cn{V_n}<2^{-n}.
\end{equation}
By \eqref{caliban}, $W_n\to0$ in cut norm and thus in cut metric; hence the
set \set{W_n} is relatively compact for the cut metric, and
$\set{W_n}\cup\set0$
is compact.

It follows from \refT{TD} (or directly from the definition) 
that the graphons $W_n$ have \ucrt.

However, $W_n$ do not have \urt.
In fact, $|W_n|\le n\qw$,
and thus for any $M$
and any set $U$ with $\leb(U)\le M$,
\begin{equation}
\norml{W_n\etta_{\qx U}}\le n\qw \leb^2(\qx U)\le M^2/n;  
\end{equation}
hence, by \eqref{sycorax}, 
\begin{equation}\label{miranda}
  \norml{W_n-W_n\etta_{\qx U}}\ge 1-M^2/n.
\end{equation}
Since $1-M^2/n\ge\frac12$ for all large $n$, \eqref{miranda} shows that
the graphons do not have \urt.

Hence, \urt{} is not necessary for compactness.

Note also that $|W_n|\le n\qw\le1$, and thus the graphons $W_n$ are all
uniformly bounded. By this and  \eqref{sycorax}, they are also \ui.
\end{example}

\subsection{Compactness, the standard case of \ps{s}}
Finally,
we consider the standard setting of graphons defined on probability spaces.
(Or, equivalently, graphons defined on \oi.)
In this setting,
it is well-known, and of fundamental importance, that the set of all
$\oi$-valued graphons is compact, as proved by \citet{LSz:Sz}.
This was extended to $L^p$-bounded and \ui{} sets of graphons on probability
spaces by 
\citet[Theorems 2.13 and  C.7]{BCCZ14a}.
We recover these results as  corollaries.

Note first that  a graphon $W$ defined on a \ps{} 
obviously is
also equivalent to graphons defined on other spaces; one example is  a trivial
extension of $W$, and
the following lemma, the proof of which is postponed to the next section,
shows that this is essentially the only possibility.

\begin{lemma}
  \label{Lps}
Let $W=(W,S,\mu)$ be a graphon.
Then the following are equivalent.
\begin{romenumerate}
\item\label{Lps=} 
$W\cong W'$ for some graphon $W'$ defined on a \ps.
\item \label{Lpstr}
$W$ is \aex{} equal to
a trivial extension of a graphon defined on a measure space
  $(S',\mu')$ with $\mu'(S')\le1$.
\item \label{LpsU}
There exists a set $U\subseteq S$ with $\mu(U)\le1$ such that $W(x,y)=0$
\aex{} on $\qxp{S}\setminus\qxp U$.
\item \label{Lpsf}
There exists a function $f:S\to\oi$ with $\int f\dd\mu\le1$ such that
\\
$f(x)f(y)W(x,y)=W(x,y)$ a.e.
\end{romenumerate}
\end{lemma}

Furthermore,
for two graphons $W_1$ and $W_2$ defined on \ps{s},
the definition by \cite{BCCH16}, see \refSS{SSdcut}, 
of the cut distance $\dcut(W_1,W_2)$
is the same 
as the usual definition for \ps{s} in \eg{} 
\cite{BR09,BCCZ14a,BCLSV1,SJ249,Lovasz}. 
Moreover, the next lemma shows that when considering limits of sequences of
graphons
on \ps{s}, it does not matter whether we require also the limit to be
defined on a \ps{} or allow it to be defined on an arbitrary \gsf{} measure
space.
In particular,  completeness and compactness properties of a set $\cW$ of
graphons on \ps{s} do not depend on whether we consider $\cW$ as a subset of
the set of all such graphons, or of all graphons on \gsf{} measure spaces
(We are more careful than usually in the statement and talk explicitly about
equivalence classes, since as
just noted, a graphon  on a \ps{} is equivalent to  graphons on other
measure spaces.)

\begin{lemma}
  \label{Loi2}
If $W_n$ are graphons defined on \ps{s}, and $W$ is a graphon such that
$W_n\to W$ in the cut metric, then there exists an equivalent graphon
$W'\cong W$ that is defined on a \ps.
In other words, for the cut metric,
  the set of equivalence classes of 
graphons defined on probability spaces is a closed subset of the set of
equivalence classes  of all graphons defined on \gsf{} measure spaces.
\end{lemma}

We postpone the proof of this lemma too to next section.

We record also a trivial fact. 
\begin{lemma}
  \label{Loi}
  Any set of graphons defined on probability spaces has \urt, and thus \ucrt.
\end{lemma}
\begin{proof}
  Take $M=1$ and $U=S$ in the definition.
\end{proof}

We return to compactness properties.
\begin{theorem}
  \label{Toi}
Let $\cW$ be  a set of graphons defined on probability spaces.
Then $\cW$ is totally bounded for the cut metric if and only if it is \ucr.

Hence, $\cW$ is relatively compact if and only if it is \ucr{} and
relatively complete.
\end{theorem}

\begin{proof}
By \refL{Loi} and  \refT{TD}.
\end{proof}

\begin{corollary}[{\cite[Theorem C.7]{BCCZ14a}}]\label{CC}
 A \ui{} set of graphons defined on probability spaces
is relatively compact for the cut metric.
\end{corollary}

\begin{proof}
A \ui{} set is \ucr, as said in \refS{SScompactsui}, so the result follows
by Theorems \ref{Toi} and \ref{TXA}.
(Or by \refL{Loi} and \refT{TUC}.)
\end{proof}
  
\begin{corollary}[{\cite[Theorem 2.13]{BCCZ14a}}]\label{CCp}
Let $1<p<\infty$ and $C<\infty$.
Then 
the set of all graphons
$W$ defined on probability spaces such that $\normlp{W}\le C$
is compact for the cut metric.
\end{corollary}

\begin{proof}
Denote this set by $\cW$.
As remarked in \refS{Scomplete},
  the set $\cW$ is \ui, and thus it is
relatively compact by \refC{CC}, so it remains
  only to show that $\cW$ is closed, \ie, that if $W_n$ is a sequence of
  graphons in $\cW$ and $W_n\to W$ in cut metric, then $W'\in\cW$ for some
  $W'\cong W$.
This follows by Lemmas \ref{LLgd} and \ref{Loi2}.
\end{proof}

\begin{remark}
  \citet[Section 2.5]{BCCZ14a} give a definition of 
\emph{$L^p$ upper regular} sequences of graphons (defined on \oi), similar
to their definition of uniformly upper regular sequences mentioned in
\refR{Ruur} above.
They use this in criteria for sequential compactness with a limit in $L^p$.
As in \refR{Ruur}, we leave it as an open problem to investigate the
relation with our notions and results.
\end{remark}

\section{Proofs of \refT{TD} and Lemmas \ref{Lps}--\ref{Loi2}}\label{ScompactPf}

We first prove a couple of technical lemmas. 
The first yields alternative (but equivalent) characterisations of the
properties uniformly [cut] regular tails.

\begin{lemma}\label{LQ2}
  \begin{thmenumerate}
	\item \label{LQ2r}
A set $\cW$ of graphons has \urt{} if
for every $\eps>0$, there exists $M<\infty$ such that for every graphon
$(W,S,\mu)\in\cW$, there exists a 
measurable function $f:S\to\oi$ such that $\int f\dd\mu\le M$
and 
\begin{equation}
  \label{aufidius}
\norml{W-(f\tensor f)W}<\eps.
\end{equation}
\item \label{LQ2cr}
A set $\cW$ of graphons has \ucrt{} if
for every $\eps>0$, there exists $M<\infty$ such that for every graphon
$(W,S,\mu)\in\cW$, there exists a 
measurable function $f:S\to\oi$ such that $\int f\dd\mu\le M$
and 
\begin{equation}
  \label{coriolanus}
\cn{W-(f\tensor f)W}<\eps.
\end{equation}
  \end{thmenumerate}
\end{lemma}
\begin{proof}
We prove \ref{LQ2cr}; the same proof works for \ref{LQ2r}.

  If \eqref{thisbe} holds, then take $f:=\etta_U$.

Conversely, suppose that \eqref{coriolanus} holds. Define
$U:=\set{x:f(x)>\frac12}$.
Then $\mu(U)\le2\norml{f}\le2M$.
Moreover,
\begin{equation}
  \begin{split}
\etta_{\Uc}(x)W(x,y)
=\sumko\etta_{\Uc}(x)\bigpar{f(x)f(y)}^k\bigpar{1-f(x)f(y)}W(x,y)
  \end{split}
\end{equation}
and thus, using \eqref{malcolm} and \eqref{coriolanus},
\begin{equation}
  \begin{split}
\bigcn{\etta_{\Uc}(x)&W(x,y)}
\le\sumko\Bigcn{\etta_{\Uc}(x)\bigpar{f(x)f(y)}^k\bigpar{1-f(x)f(y)}W(x,y)}
\\&
\le\sumko\sup_x\bigpar{\etta_{\Uc}(x)f(x)^k}\sup_y\bigpar{ f(y)^k}
\Bigcn{\bigpar{1-f(x)f(y)}W(x,y)}
\\&
<\sumko2^{-k}\eps=2\eps.
  \end{split}
\raisetag{1.5\baselineskip}
\end{equation}
By symmetry also 
$\bigcn{\etta_{\Uc}(y)W(x,y)}<2\eps$, and thus
\begin{equation}
  \begin{split}
	\bigcn{W-\etta_{U\times U}W}
&\le \bigcn{\etta_{\Uc}(x)W(x,y)}
+ \bigcn{\etta_U(x)\etta_{\Uc}(y)W(x,y)}
\\& 
\le \bigcn{\etta_{\Uc}(x)W(x,y)}
+ \bigcn{\etta_{\Uc}(y)W(x,y)}
\\&
<4\eps.
 \end{split}
\end{equation}
Consequently, \eqref{thisbe} holds, with $(M,\eps)$ replaced by $(2M,4\eps)$.
\end{proof}

We say that two sets of graphons are equivalent if every graphon in one of
the sets is equivalent to some graphon in the other set.

\begin{lemma}\label{LQ3}
  Let $\cW$ and $\cW'$ be two equivalent sets of graphons.
  \begin{romenumerate}
  \item \label{LQ3r}
If\/ $\cW$ has \urt, then so has $\cW'$, and conversely.
  \item \label{LQ3cr}
If\/ $\cW$ has \ucrt, then so has $\cW'$, and conversely.
  \end{romenumerate}
\end{lemma}
A special case of \ref{LQ3r} is given in \cite[Lemma 6.1]{BCCH16}.
\begin{proof}
  Again, the same proof works for both parts; we choose \ref{LQ3cr}.

We show that for every $\eps$ and $M$,
if $W$ and $W'$ are two equivalent graphons and one of them satisfies
\eqref{coriolanus} for 
some $f$ as in \refL{LQ2}, there so does the other.
The result then follows by \refL{LQ2}.

By \refT{T=}, it suffices to consider the case when $W$ and $W'$ are \eeq.
The case of a trivial extension is trivial, and thus it
suffices to consider the case of a pull-back $W'=W\qgf$ for some \mpm{}
$\gf:(S_1,\cF_1,\mu_1)\to (S_2,\cF_2,\mu_2)$.

First, if \eqref{coriolanus} holds, then using \eqref{dunsinane},
\begin{equation}
\cnx{\mu_1}{W\qgf-(f\qgf\tensor f\qgf)W\qgf}
=\cnx{\mu_2}{W-(f\tensor f)W}<\eps,
\end{equation}
and $f\qgf:S_1\to\oi$ with $\int f\qgf\dd\mu_1=\int f\dd\mu_2\le M$.

Conversely, suppose that
$\cnx{\mu_1}{W\qgf-(f\tensor f)W\qgf}<\eps$ for some $f:S_1\to\oi$ with
$\int_{S_1} f\le M$.
Let $\cF':=\gf\qw(\cF_2):=\set{\gf\qw(A):A\in\cF_2}$, and let $f':=\E(f\mid
\cF')$. (Although conditional expectations usually are defined for
probability spaces only, there is no problem to extend the definition to
\gsf{} measure spaces, for example by considering a partition $S=\bigcup_k
  S_k$ into subsets $S_k\in\cF'$ with finite measure.)
Since $f'$ is measurable for $\cF'=\gf\qw(\cF_2)$, $f'=g\qgf$ for some
$g:S_2\to\oi$
with 
\begin{equation}
  \int_{S_2}g\dd\mu_2
=\int_{S_1}g\qgf\dd\mu_1=\int_{S_1}f'\dd\mu_1=\int_{S_1}f\dd\mu_1\le M.
\end{equation}
Furthermore, $W\qgf$ is $\qxp{\cF'}$-measurable, and thus
$\E\bigpar{(f\tensor f)W\qgf\mid \qx{\cF'}}=(f'\tensor f')W\qgf$.
Consequently,
\begin{equation}
  \begin{split}
  \cnx{\mu_2}{W-(g\tensor g)W}
&=
  \cnx{\mu_1}{\bigpar{W-(g\tensor g)W}\qgf}	
\\&
=
  \cnx{\mu_1}{W\qgf-(f'\tensor f')W\qgf}
\\&
=
  \cnx{\mu_1}{\E\bigpar{W\qgf-(f\tensor f)W\qgf\mid \cF'\times\cF'}}
\\&
\le
  \cnx{\mu_1}{W\qgf-(f\tensor f)W\qgf}
<\eps
  \end{split}
\end{equation}
where the last but one inequality easily follows from \eqref{macduff}.
\end{proof}

\begin{proof}[Proof of \refT{TD}]
  \ref{TDtb}$\implies$\ref{TD2}.
First, note that \refL{LQ3} shows that if we replace $\cW$ by an equivalent
set of graphons, then \ref{TD2} is preserved. 
Thus, using \refP{P2.8}
we may assume that every graphon $W$ in $\cW$ is defined
on $\bbRp$.

Let $\eps>0$. By assumption, there exists a finite $\eps$-net
$\set{W_i}_{i=1}^N$ for $\cW$; we may assume that also every $W_i$ is
defined on $\bbRp$.
Thus, if $W\in\cW$, then there exists $W_i$ such that $\dcut(W,W_i)<\eps$,
and thus by \refP{P4.3c} a \mpb{} $\gf:\bbRp\to\bbRp$
such that 
\begin{equation}
  \label{antonio}
\cn{W-W_i\qgf}<\eps.
\end{equation}

Next,  every $W_i$ is integrable, and thus we may find a set $U_i$ of finite
measure such that $\cn{W_i-W_i\ettaqx{U_i}}\le\norml{W_i-W_i\ettaqx{U_i}}<\eps$
 and a number $B_i$ such that if $\bW_i:=W_i\etta_{|W_i|\le B_i}$, then 
$\cn{W_i-\bW_i}\le\norml{W_i-\bW_i}<\eps$.
Let $M:=\max_{i\le N}\mu(U_i)$ and $B:=\max_{i\le N} B_i$.

If \eqref{antonio} holds, then
let $U:=\gf\qw(U_i)$. Then 
$W_i\qgf\ettaqx U=W_i\qgf\ettaqx{U_i}\qgf=(W_i\ettaqx{U_i})\qgf$ and thus
\begin{equation}
  \begin{split}
&\cn{W-W\ettaqx U}
\\&\qquad
\le\cn{W-W_i\qgf}+\cn{(W_i-W_i\ettaqx{U_i})\qgf}+\cn{(W_i\qgf-W)\ettaqx{U}}
\\&\qquad
<\eps+\cn{W_i-W_i\ettaqx{U_i}}+\eps
< 3\eps.
  \end{split}
\raisetag\baselineskip
\end{equation}
Similarly,
\begin{equation}
  \cn{W-\bW_i\qgf}
\le  \cn{W-W_i\qgf}+ \cn{(W_i-\bW_i)\qgf}
<\eps+\cn{W_i-\bW_i}<2\eps.
\end{equation}
Since $\leb(U)=\leb(U_i)\le M$ and $|\bW_i\qgf|\le B_i\le B$,
it follows that $\cW$ has \ucrt{} and is \ucr.

  \ref{TD2}$\implies$\ref{TD1}.
Given $\eps>0$, let $M$ and $B$ be as in the definitions of \ucrt{} and
\ucr; thus if $W=(W,S,\mu)\in\cW$, there exist $U\subseteq S$ with
$\mu(U)\le M$ such that \eqref{thisbe} holds and a graphon $V$ on $S$ with
$|V|\le B$ such that \eqref{quince} holds.
Let $V':=V\etta_{\qx U}$. Then
$|V'|\le B\ettaqx U$ and
\begin{equation}
  \begin{split}
	  \cn{W-V'}&
\le \cn{W-W\etta_{\qx U}}+\cn{(W-V)\ettaqx{U}}
\\&
\le \cn{W-W\etta_{\qx U}}+\cn{W-V}<2\eps.
  \end{split}
\end{equation}
Thus \ref{TD1} follows.


\ref{TD1}$\implies$\ref{TDtb}.
Let $\eps>0$ and let $M$ and $B$ be as in \ref{TD1}.
Let $\cW(M, B)$ be the set of all graphons $(V,S,\mu)$ with $|V|\le B$
and $\mu(S)= M$.
If $W$, $U$ and $V$ are as in \ref{TD1}, then $V$ is a trivial extension of the
restriction $V_U:=V\restr{\qx U}$, which is defined on $U$ with $\mu(U)\le
M$, and thus there exists 
a trivial extension $V_U'$ of $V_U$ to a measure space $(S',\mu')$ with
$\mu'(S')=M$. Then $V_U'\in\cW(M,B)$; furthermore,
\begin{equation}\label{toby}
\dcut(W,V_U')=
  \dcut(W,V_U)=\dcut(W,V)\le\cn{W-V}<\eps.
\end{equation}

Next, the mapping $\Psi:V\mapsto\bigpar{\stri{V}{M\qww}+B}/(2B)$ maps $\cW(M,B)$
bijectively onto the set of \oi-valued graphons defined on probability
spaces, see \eqref{str1}, 
and $\Psi$ is a homeomorphism for the cut metric. (In fact, 
$\dcut(V,V')=2BM^2\dcut(\Psi(V),\Psi(V'))$ for all $V,V'\in\cW(M,B)$.)
As is well-known from the standard theory of graphons on probability spaces
(or on $\oi$), the latter set is compact for the cut metric
\cite{LSz:Sz}. Consequently, $\cW(M,B)$ is compact for the cut metric and
thus totally bounded. Hence, there exists a finite $\eps$-net
$\set{V_i}_{i=1}^N$ for $\cW(M,B)$.

We have shown in \eqref{toby} that if $W\in\cW $, then there exists
a graphon $V'_U\in\cW(M,B)$ such that $\dcut(W,V'_U)<\eps$; furthermore, since
$\set{V_i}$ is an $\eps$-net for $\cW(M,B)$, there exists $V_i$ such that
$\dcut(V'_U,V_i)<\eps$. Hence, $\dcut(W,V_i)<2\eps$, and it follows that
$\set{V_i}_1^N$ is a finite $2\eps$-net for $\cW$.
Since $\eps$ is arbitrary, $\cW$ is totally bounded.
\end{proof}

\begin{proof}[Proof of \refL{Lps}]
\ref{Lpstr}$\iff$\ref{LpsU}.
Clear by the definition of trival extension.

\ref{LpsU}$\implies$\ref{Lpsf}.
Take $f:=\etta_U$.

\ref{Lpsf}$\implies$\ref{LpsU}.
Take $U:=\set{x:f(x)=1}$.

\ref{Lpstr}$\implies$\ref{Lps=}.
If $W$ is a trivial extension of $(W',S',\mu')$ with $\mu'(S')\le1$, then
$W'$ has a trivial extension $W''$ defined on a probability space, and
$W\cong W'\cong W''$.

\ref{Lps=}$\implies$\ref{Lpsf}.
This follows by the same proof as for \refL{LQ3}, with $M=1$ and '$<\eps$'
replaced by '$=0$' (or, equivalently, for all $\eps$ simultaneously),
recalling \eqref{cawdor}.
\end{proof}

\begin{proof}[Proof of \refL{Loi2}]
As usual, we may 
by \refP{P2.8} replace the 
graphons by equivalent graphons on $(\bbRp,\leb)$ and  assume that
$W$ and every $W_n$ is defined on $\bbRp$.
By \refP{P4.3c}, this implies the existence of
pull-backs $W_n':=W_n\qphi n\cong W_n$ such that $\cn{W-W_n'}
< \dcut(W,W_n)+1/n\to0$ as \ntoo.

By \refL{Lps}, there exist sets $U_n\subset\bbRp$ with $\leb(U_n)\le1$ such
that $W_n'=W_n'\ettaqx{U_n}$ a.e.
Hence,
\begin{equation}\label{rosalind}
  \cn{W-W\ettaqx{U_n}}
\le
  \cn{W-W_n'}+  \cn{\ettaqx{U_n}(W_n'-W)}
\le 2\cn{W-W_n'}\to0,
\end{equation}
as \ntoo.

The unit ball of $L^\infty(\bbRp)=L^1(\bbRp)^*$ is weak-$*$ compact and
metrizable (since $L^1(\bbRp)$ is separable) 
\cite[Theorems V.4.2 and V.5.1]{DS}.
Hence, by considering a subsequence, we may assume that $\etta_{U_n}\wxto f$
for some $f\in L^\infty(\bbRp)$ with $|f|\le1$; furthermore, this implies
$f:\bbRp\to\oi$. For any $T\subset\bbRp$ with $\leb(T)<\infty$, 
$\int_T f=\int_{\bbRp} \etta_T f=\lim\int_{\bbRp}\etta_T\etta_{U_n}\le1$;
hence, by monotone 
convergence, $\int_{\bbRp} f\le1$.

Moreover.
it follows, since $L^1(\bbRp)\tensor L^1(\bbRp)$ is dense in $L^1(\bbRp^2)$,
that $\ettaqx{U_n}=\etta_{U_n}\tensor\etta_{U_n}\wxto f\tensor f$ in
$L^\infty(\bbRp^2)$.
Hence, for any measurable sets $T,U\subseteq\bbRp$,
\begin{equation}
  \begin{split}
	  \int_{T\times U} W\ettaqx{U_n}&
=
\int_{\bbRp^2} W\etta_{T\times U} \ettaqx{U_n}
\\&
\to
\int_{\bbRp^2} W\etta_{T\times U} (f\tensor f)
=\int_{T\times U}(f\tensor f) W.
  \end{split}
\end{equation}
On the other hand, \eqref{rosalind} implies
\begin{equation}
  \int_{T\times U}\bigpar{W\ettaqx{U_n}-W}\to0.
\end{equation}
Consequently,
$\int_{T\times U}(f\tensor f) W =\int_{T\times U}W$ for all $T,U\subseteq
S$, and thus, by \eqref{cn}, 
$\cn{(f\tensor f)W-W}=0$, so by \eqref{cawdor},  $(f\tensor f)W=W$ a.e.
The result follows by \refL{Lps}\ref{Lpsf}$\implies$\ref{Lps=}.
\end{proof}

This completes the proof of the results in
\refS{Scompact}.

\section*{Acknowledgement}
This work was partly carried out during a visit to the 
Isaac Newton Institute for Mathematical Sciences
during the programme 
Theoretical Foundations for Statistical Network Analysis in 2016
(EPSCR Grant Number EP/K032208/1)
and was partially supported by a grant from the Simons foundation, and
a grant from
the Knut and Alice Wallenberg Foundation.
I also thank Christian Borgs, Henry Cohn, Jennifer Chayes, 
Nina Holden,
Laci Lovasz
and Daniel Roy for helpful
discussions and comments that have improved the paper.

\newcommand\AAP{\emph{Adv. Appl. Probab.} }
\newcommand\JAP{\emph{J. Appl. Probab.} }
\newcommand\JAMS{\emph{J. \AMS} }
\newcommand\MAMS{\emph{Memoirs \AMS} }
\newcommand\PAMS{\emph{Proc. \AMS} }
\newcommand\TAMS{\emph{Trans. \AMS} }
\newcommand\AnnMS{\emph{Ann. Math. Statist.} }
\newcommand\AnnPr{\emph{Ann. Probab.} }
\newcommand\CPC{\emph{Combin. Probab. Comput.} }
\newcommand\JMAA{\emph{J. Math. Anal. Appl.} }
\newcommand\RSA{\emph{Random Struct. Alg.} }
\newcommand\ZW{\emph{Z. Wahrsch. Verw. Gebiete} }

\newcommand\AMS{Amer. Math. Soc.}
\newcommand\Springer{Springer-Verlag}
\newcommand\Wiley{Wiley}

\newcommand\vol{\textbf}
\newcommand\jour{\emph}
\newcommand\book{\emph}
\newcommand\inbook{\emph}
\def\no#1#2,{\unskip#2, no. #1,} 
\newcommand\toappear{\unskip, to appear}

\newcommand\arxiv[1]{\texttt{arXiv:#1}}
\newcommand\arXiv{\arxiv}

\def\nobibitem#1\par{}

\end{document}